\documentclass[aop, preprint, 12pt]{imsart}

\RequirePackage[OT1]{fontenc}
\RequirePackage{amsthm,amsmath}
\RequirePackage[numbers]{natbib}
\RequirePackage[colorlinks,citecolor=blue,urlcolor=blue]{hyperref}
\usepackage{mathrsfs}
\usepackage{amsfonts}
\usepackage{color}
\usepackage{pifont}
\usepackage{amssymb}
\usepackage[english]{babel}
\usepackage[T1]{fontenc}
\usepackage{graphicx}
\usepackage{subfigure}
\usepackage{latexsym}
\usepackage{amstext}
\usepackage{mathrsfs}
\usepackage{hyperref}
\usepackage{dsfont}
\usepackage{graphics}
\usepackage{enumerate}
\usepackage{paralist}
\usepackage{color}
\usepackage{fullpage}
\newtheorem{prop}{Proposition}[section]

\newtheorem{lem}[prop]{Lemma}

\numberwithin{equation}{section}

\newcommand{\beq}{\begin{eqnarray}}
\newcommand{\beqq}{\begin{eqnarray*}}
\newcommand{\eeq}{\end{eqnarray}}
\newcommand{\eeqq}{\end{eqnarray*}}

\usepackage{mathtools}
\usepackage[english]{babel}
\usepackage[utf8]{inputenc}
\usepackage{amsmath}
\usepackage{graphicx}
\usepackage[colorinlistoftodos]{todonotes}
\usepackage{polynom}
\usepackage{amsfonts}
\usepackage{dsfont}
\usepackage{amsthm}
\usepackage{hyperref}
\usepackage{graphicx}
\newtheorem{theorem}{Theorem}[section]

\newtheorem{corollary}[theorem]{Corollary}

\usepackage{fullpage}
\usepackage[T1]{fontenc}
\usepackage{hyperref}
\usepackage{xcolor}
\definecolor{link-color}{rgb}{0.15,0.4,0.15}
\hypersetup{
  colorlinks,
  linkcolor = {link-color}, citecolor = {link-color},
}

\usepackage{graphicx}
\usepackage{hyperref}
\usepackage{amsmath,amsthm,amssymb}
\usepackage{bbm}
\usepackage{tabularx}
\usepackage{tikz}

\newcommand{\N}{\mathbb{N}}
\renewcommand{\P}{\mathbb{P}}
\newcommand{\E}{\mathbb{E}}

    \def\d{{\textnormal d}}

\newenvironment{eqnarr}{\begin{IEEEeqnarray}{rCl}}{\end{IEEEeqnarray}\ignorespacesafterend}
\renewcommand{\eqref}[1]{\hyperref[#1]{(\ref*{#1})}}

\newcommand*{\norm}[1]{\lVert #1 \rVert}

    \def\beq{\begin{eqnarr}}%arw
    \def\eeq{\end{eqnarr}}%arw
    \def\beqq{\begin{eqnarray*}} %arw
    \def\eeqq{\end{eqnarray*}} %arw

        \def\d{{\rm d}}

    \def\d{{\textnormal d}}
\newtheorem{remark}{Remark}[section]

\newcommand*{\pref}[1]{\hyperref[#1]{(\ref*{#1})}}
\newcommand*{\refpref}[2]{\hyperref[#2]{\ref*{#1}(\ref*{#2})}}

\everymath{\displaystyle}

\makeatletter
\def\namedlabel#1#2{\begingroup
    #2%
    \def\@currentlabel{#2}%
    \phantomsection\label{#1}\endgroup
}
\makeatother

  \newcommand{\D}{{\rm d}}

\newcommand{\U}{\texttt U}

\newcommand{\tk}{\texttt k}
\newcommand{\etk}{\emph{\texttt k}}

\newcommand{\Q}{\texttt Q}

\newcommand{\bP}{\texttt P}
\newcommand{\ebP}{\emph{\texttt P}}

\newcommand{\bS}{{\texttt S}}

\newcommand{\bF}{{\texttt F}}

\newcommand{\bL}{{\texttt L}}

\startlocaldefs
\numberwithin{equation}{section}
\theoremstyle{plain}

\endlocaldefs

\begin{document}

\begin{frontmatter}
\title{Stochastic Methods for the Neutron  Transport Equation II: Almost sure growth}
%continuous-state branching processes}

\runtitle{Recurrent extension for ssMp in a Wedge}
%\thankstext{T1}{Footnote to the title with the ``thankstext'' command.}

\begin{aug}
%\author{\fnms{Alexander M.G. Cox}\thanksref{t3}\ead[label=e1]{a.m.g.cox@bath.ac.uk}}, 
\author{\fnms{Simon C. Harris}\thanksref{t3}\ead[label=e2]{simon.harris@auckland.ac.nz}},
\author{\fnms{Emma Horton}\thanksref{t1}\ead[label=e5]{emma.horton94@gmail.com}},
 \author{\fnms{Andreas E. Kyprianou}\thanksref{t3}\ead[label=e3]{a.kyprianou@bath.ac.uk}}
%\and \author{\fnms{Denis Villemonais}\ead[label=e6]{denisvillemonais@gmail.com}}
%\and 
%\author{\fnms{Minmin Wang}\thanksref{t3}\ead[label=e4]{m.wang@bath.ac.uk}},

\thankstext{t1}{This work was conducted whilst Emma Horton was supported by a PhD scholarship from industrial partner Wood (formerly Amec Foster Wheeler) at the University of Bath.}

\thankstext{t3}{Supported by EPSRC grant EP/P009220/1.}

%\thankstext{t2}{Supported by a Royal Thai PhD scholarship}
%\thankstext{t3}{Second supporter of the project}
%\runauthor{F. Author et al.}

%\affiliation{University of Bath, University of Auckland, \\University of Bath and University of Bath}

\address{
%A. M. G. Cox,  \\
E. L. Horton,\\
Institut \'Elie Cartan de Lorraine \\
Universit\'e de Lorraine\\
54506 Vandoeuvre-l\`es-Nancy Cedex, \\
France.\\
\printead{e5}
}

\address{
%A. M. G. Cox,  \\
%E. L. Horton,\\
A.E. Kyprianou \\
Department of Mathematical Sciences \\
University of Bath\\
Bath, BA2 7AY\\
 UK.\\
%\printead{e1}\\
%\printead{e5}\\
\printead{e3}\\
%\printead{e4}
%\\
%\phantom{E-mail:\ }\printead*{e2}
}

\address{S. C. Harris\\
Department of Statistics\\
University of Auckland\\
Private Bag 92019\\
Auckland 1142\\
New Zealand\\
\printead{e2}
}

%\address{D. Villemonais\\
%Institut \'Elie Cartan de Lorraine\\
%Bureau 123\\
%Universit\'e de Lorraine \\
%54506\\ 
%Vandoeuvre-l\`es-Nancy Cedex\\
%France
%\printead{e6}
%}
\end{aug}

\begin{abstract}\hspace{0.1cm}
The neutron transport equation (NTE) describes the flux of neutrons across a planar cross-section  in an inhomogeneous fissile medium when the process of nuclear fission is active. Classical work on the NTE emerges from the applied mathematics literature in the 1950s through the work of R. Dautray and collaborators, \cite{D, DL6, M-K}. The NTE also has a probabilistic representation through the semigroup of the underlying {\it physical process} when envisaged as a stochastic process; cf. \cite{D, LPS, MT, MWY}.  More recently, \cite{MultiNTE} and \cite{SNTE} have continued the probabilistic analysis of the NTE, introducing more recent ideas from the theory of spatial branching processes and quasi-stationary distributions. In this paper, we continue in the same vein and look at a fundamental description of stochastic growth in the supercritical regime. Our main result provides a significant improvement on the last known contribution to growth properties of the physical process in  \cite{MWY}, bringing neutron transport theory in line with modern branching process theory such as \cite{HHK, EKW}. An important aspect of the proofs focuses on the use of a skeletal path decomposition, which we derive for general branching particle systems in the new context of non-local branching generators.

\end{abstract}



\end{frontmatter}

\setcounter{tocdepth}{1}
%\tableofcontents

\section{Introduction}\label{intro} In this article we continue our previous work in \cite{SNTE} and look in more detail at the stochastic analysis of the Markov process that lies behind the Neutron Transport Equation (NTE). We recall that the  latter describes the flux, $\Psi_t$, at time $t\geq 0$, of neutrons across a planar cross-section  in an inhomogeneous fissile medium (measured in number of neutrons per cm$^2$ per second). Neutron flux  is described in terms of  the configuration variables $ (r, v) \in D \times V$, where $D\subseteq\mathbb{R}^3$ is (in general) a non-empty, smooth, open, bounded and convex domain such that $\partial D$ has zero Lebesgue measure, and 
$V$ is the velocity space, which is given by $V = \{\upsilon\in \mathbb{R}^3:\upsilon_{\texttt{min}}\leq  |\upsilon|\leq \upsilon_{\texttt{max}}\}%\mathbb{S}_2 \times [\upsilon_{\texttt{min}}, \upsilon_{\texttt{max}}]
$, where $0<\upsilon_{\texttt{min}}<\upsilon_{\texttt{max}}<\infty$. 

\smallskip
In its backwards form, the NTE is introduced as an integro-differential equation of the form 
\begin{align}
\frac{\partial}{\partial t}\psi_t(r, \upsilon) &=\upsilon\cdot\nabla\psi_t(r, \upsilon)  -\sigma(r, \upsilon)\psi_t(r, \upsilon)\notag\\
&+ \sigma_{\texttt{s}}(r, \upsilon)\int_{V}\psi_t(r, \upsilon') \pi_{\texttt{s}}(r, \upsilon, \upsilon')\d\upsilon' + \sigma_{\texttt{f}}(r, \upsilon) \int_{V}\psi_t(r, \upsilon') \pi_{\texttt{f}}(r, \upsilon, \upsilon')\d\upsilon',
\label{bNTE}
\end{align}
where the five fundamental quantities  $\sigma_{\texttt{s}}$, $\pi_{\texttt{s}}$, $\sigma_{\texttt{f}}$, $\pi_{\texttt{f}}$ and $\sigma$ (known as {\it cross-sections} in the physics literature) are all uniformly bounded and measurable with the following interpretation:
\begin{align*}
\sigma_{\texttt{s}}(r, \upsilon) &: \text{ the rate at which scattering occurs from incoming velocity $\upsilon$ at position $r$,}\\
\sigma_{\texttt{f}}(r, \upsilon) &: \text{  the rate at which fission occurs from incoming velocity $\upsilon$  at position $r$,}\\
\sigma(r, \upsilon) &: \text{ the sum of the rates } \sigma_{\texttt{f}}+ \sigma_{\texttt{s}} \text{ and is known as the total cross section,}\\
\pi_{\texttt{s}}(r, \upsilon, \upsilon') &: \text{  probability density that an incoming velocity $\upsilon$ at position $r$ scatters to an } \\
 &\hspace{0.5cm}\text{outgoing velocity, with probability $\upsilon'$ satisfying }\textstyle{\int_V}\pi_{\texttt{s}}(r, \upsilon, \upsilon'){\rm d}\upsilon'=1,\text{ and }\\
 \pi_{\texttt{f}}(r, \upsilon, \upsilon') &:  \text{  density of expected neutron yield at velocity $\upsilon'$ from fission with }   \\
 &\hspace{0.5cm}\text{incoming velocity  $\upsilon$ satisfying } \textstyle{\int_V\pi_{\texttt{f}}}(r, \upsilon, \upsilon')\d\upsilon' <\infty.
\end{align*}

It is also usual to assume the additional boundary conditions 
\begin{equation}
\left\{
\begin{array}{ll}
\psi_0(r, \upsilon) = g(r, \upsilon) &\text{ for }r\in  D, \upsilon\in{V},
\\
&
\\
\psi_t(r, \upsilon)= 0& \text{ for } t \ge 0 \text{ and } r\in \partial D,
\text{ if }\upsilon
\cdot{\bf n}_r>0,
\end{array}
\right.
\label{BC}
\end{equation}
where  ${\bf n}_r$ is the outward facing normal of $D$ at $r\in \partial D$ and $g: D\times {V}\to [0,\infty)$ is a bounded, measurable function which we will later assume has some additional properties. Roughly speaking, this means that neutrons at the boundary which are travelling in the direction of the exterior of the domain are lost to the system. 
\smallskip

{\color{black}We will also work with some of (but not necessarily all of) the following assumptions in our results:
\smallskip
\begin{itemize}
{\bf \item[(H1)] Cross-sections $\sigma_{\texttt{s}}$, $\sigma_{\texttt{f}}$, $\pi_{\texttt{s}}$ and $\pi_{\texttt{f}} $ are uniformly bounded away from infinity.
\vspace{0.2cm}
\item[(H2)] We have $\sigma_{\texttt{s}} \pi_{\texttt{s}}  + 
\sigma_{\texttt{f}} \pi_{\texttt{f}}>0$ on $D\times V\times V$.
\vspace{0.2cm}
\item[(H3)] There is an open ball $B$ compactly embedded in $D$ such that $\sigma_{\texttt{f}}\pi_{\texttt{f}} >0$ on $B\times V\times V$.
\vspace{0.2cm}
\item[(H4)] Fission offspring are bounded in number  by the constant $n_{\texttt{max}}> 1$.}
\end{itemize}

\smallskip

We note that these assumptions are sufficient but not necessary, and refer the reader to Remark 2.1 in~\cite{SNTE} for a discussion of their implications.}

\subsection{Rigorous interpretation of the NTE} As explained in the companion paper \cite{SNTE},
the NTE \eqref{bNTE} is not a  meaningful equation in the pointwise sense. Whereas previously \eqref{bNTE} has been interpreted as an abstract Cauchy process on the $L_2(D\times V)$ space, for probabilistic purposes, the NTE can be better understood in its mild form; see the review discussion in \cite{MultiNTE}.
In particular,  the NTE is henceforth understood as the unique bounded solution on bounded intervals of time which satisfy \eqref{BC} and  the so-called {\it mild equation}
\begin{equation}
\psi_t[g](r,\upsilon) = {\texttt{U}}_t[g](r,\upsilon) + \int_0^t {\texttt{U}}_s[({\bS} + {\bF})\psi_{t-s}[g]](r,\upsilon)\d s, \qquad t\geq 0, r\in D, \upsilon\in V.
\label{mild}
\end{equation}
for $g\in L^+_\infty(D\times V)$, the space of non-negative functions in $L_\infty(D\times V)$. In \eqref{mild}, the 
advection semigroup  is given by 
\begin{equation}
\texttt{U}_t[g]( r,\upsilon) = g( r+\upsilon t, \upsilon)\mathbf{1}_{(t<\kappa^D_{r,\upsilon})}, \qquad t\geq 0.
\label{adv}
\end{equation}
where $
\kappa_{r,\upsilon}^{D} := \inf\{t>0 : r+\upsilon t\not\in D\}, 
$
the scattering operator is given by
\begin{equation}
\bS g(r, \upsilon) = \sigma_{\texttt s}(r, \upsilon)\int_V g(r,\upsilon')\pi_{\texttt s}(r,\upsilon,\upsilon')\d \upsilon' - \sigma_{\texttt s}(r,\upsilon)g(r,\upsilon) ,
\label{S}
\end{equation}
and the fission operator is given by
\begin{equation}
\bF g(r, \upsilon)  = \sigma_{\texttt f}(r,\upsilon)\int_V g(r,\upsilon')  \pi_{\texttt f}(r,\upsilon,\upsilon')\d \upsilon'- \sigma_{\texttt f}(r,\upsilon)g(r,\upsilon),
\label{F}
\end{equation}
for $r\in D$, $\upsilon \in V$ and $g\in L^+_\infty(D\times V)$. 

\smallskip

The papers \cite{SNTE} and \cite{MultiNTE} discuss in further detail how the mild representation relates to the other classical representation of the NTE via an abstract Cauchy problem which has been treated in e.g. \cite{D, DL6, M-K}. To understand better why the mild equation \eqref{mild} is indeed a suitable representation fo the NTE, we need to understand the probabilistic model of the physical process of nuclear fission.

\begin{figure}[h!]
\includegraphics[width=0.5 \textwidth]{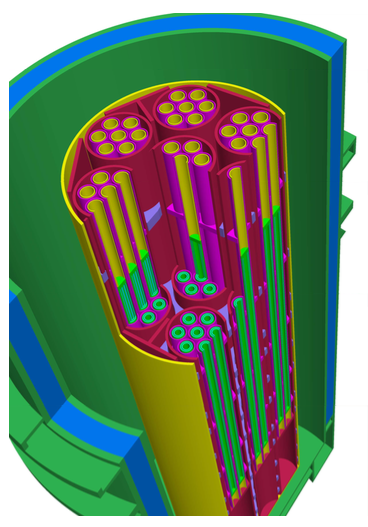}
\caption{The geometry of a nuclear reactor core representing a physical domain $D$, on to which the different cross-sectional values of  $ \sigma_{\emph{\texttt{s}}}, \sigma_{\emph{\texttt{f}}},  \pi_{\emph{\texttt{s}}},  \pi_{\emph{\texttt{f}}}$ as mapped as numerical values.}
\label{fig}
\end{figure}
\subsection{Neutron Branching Process}\label{NBPsect} Let us recall from \cite{SNTE}, the {\it neutron branching process} (NBP), whose expectation semigroup provides the solution to \eqref{mild}. It is modelled as  a branching process, which at time $t\geq0$ is represented by a configuration of particles which are specified via their physical location and velocity in $D\times V$, say $\{(r_i(t), \upsilon_i(t)): i = 1,\dots , N_t\}$, where $N_t$ is the number of particles alive at time $t\ge 0$.
In order to describe the process, we will represent it as a process in the space of finite atomic measures
\begin{equation}
X_t(A) = \sum_{i=1}^{N_t}\delta_{(r_i(t), \upsilon_i(t))}(A), \qquad A\in\mathcal{B}(D\times V), \;t\ge 0,
\label{atomicvalued}
\end{equation}
%
%  Comment
%
%\marginpar{\scriptsize\color{red} Why not using Ulam-Harris labelling throughout? eg. replace $i=1,\dots,N_t$ with $i\in \mathcal N_t$, and $N_t:=|\mathcal N_t|$}
%
%
%
where $\delta$ is the Dirac measure, defined on $\mathcal{B}(D\times V)$, the Borel subsets of $D\times V$.
The evolution of $(X_t, t\geq 0)$ is a stochastic process valued in the space of  measures
$
\mathcal{M}(D\times V): = \{\textstyle{\sum_{i = 1}^n}\delta_{(r_i,\upsilon_i)}: n\in \mathbb{N}, (r_i,\upsilon_i)\in D\times V, i = 1,\cdots, n\}
$
which evolves randomly as follows.

\smallskip

A particle positioned at $r$ with velocity $\upsilon$ will continue to move along the trajectory $r + \upsilon t$, until one of the following things happen. 
\begin{enumerate}
\item[(i)] The particle leaves the physical domain $D$, in which case it is instantaneously killed. 

\item[(ii)] Independently of all other neutrons, a scattering event occurs when a neutron comes in close proximity to an atomic nucleus and, accordingly, makes an instantaneous change of velocity. For a neutron in the system with position and velocity $(r,\upsilon)$, if we write $T_{\texttt{s}}$ for the random time that scattering may occur, then  providing $r+\upsilon t\in D$, independently of the action of fission, 
$
\Pr(T_{\texttt{s}}>t) = \exp\{-\textstyle{\int_0^t} \sigma_{\texttt{s}}(r+\upsilon s, \upsilon){\rm d}s \}, $ for $t\geq0.$
\smallskip

When scattering occurs at space-velocity $(r,\upsilon)$, the new velocity is selected in $V$ independently with probability $\pi_{\texttt{s}}(r, \upsilon, \upsilon')\d\upsilon'$. 

\item[(iii)] Independently of all other neutrons, a fission event occurs when a neutron smashes into an atomic nucleus. 
For a neutron in the system  with initial position and velocity $(r,\upsilon)$, if we write $T_{\texttt{f}}$ for the random time that fission may occur, then, providing $r+\upsilon t\in D$, independently scattering, 
$
\Pr(T_{\texttt{f}}>t) = \exp\{-\textstyle{\int_0^t} \sigma_{\texttt{f}}(r+\upsilon s, \upsilon){\rm d}s \},$ for $t\geq 0.
$
\smallskip

When fission occurs,  the smashing of the atomic nucleus produces  lower mass isotopes and releases  a random number of neutrons, say $N\geq 0$, which are ejected from the point of impact with randomly distributed, and possibly correlated, velocities, say $\{\upsilon_i: i=1,\cdots, N\}$. The outgoing velocities are described by  the atomic random measure 
\begin{equation}
\label{PP}
\mathcal{Z}(A): = \sum_{i= 1}^{N } \delta_{\upsilon_i}(A), \qquad A\in\mathcal{B}(V).
\end{equation} 

When fission occurs at location $r\in\mathbb{R}^d$ from a particle with incoming velocity $\upsilon\in{V}$, 
we denote by ${\mathcal P}_{(r,\upsilon)}$ the law of $\mathcal{Z}$.
The probabilities ${\mathcal P}_{(r,\upsilon)}$ are such that, for  $\upsilon'\in{V}$, for bounded and measurable $g: V\to[0,\infty)$,
\begin{align}
\int_V g(\upsilon')\pi_{\texttt{f}}(r, v, \upsilon')\d\upsilon' &= {\mathcal E}_{(r,\upsilon)}\left[\int_V g(\upsilon')\mathcal{Z}(\d \upsilon')\right]
=: {\mathcal E}_{(r,\upsilon)}[\langle g, \mathcal{Z}\rangle],
\label{Erv}
\end{align}
where $ {\mathcal E}_{(r,\upsilon)}$ denotes expectation with respect to $ {\mathcal P}_{(r,\upsilon)}$.
Note, the possibility that $\Pr(N = 0)>0$, which will be tantamount to neutron capture (that is, where a neutron slams into a nucleus but no fission results and the neutron is absorbed into the nucleus). 
\end{enumerate}
\smallskip

%In all cases, it is natural (cf. \cite{SNTE}) to make the  assumption that:
% 
% \begin{center}
% {\bf (H2): Fission offspring are bounded in number  by the constant $n_{\texttt{max}}> 1$.}
% \end{center}
%
%
%
%In particular this means that 
%\[
%\sup_{ r\in D, \upsilon\in V}\int_V\pi_{\texttt{f}}(r,\upsilon,\upsilon')\d \upsilon'\leq n_\texttt{max}.
%\]

Write $\mathbb{P}_{\mu}$ for the the law of $X$ when issued from an initial configuration $\mu\in\mathcal{M}(D\times V)$.
Coming back to how the physical process relates to the NTE, it was show in 
\cite{MultiNTE, SNTE, D, DL6} that, under the assumptions (H1) and (H2), the unique solution, which is bounded on bounded intervals of time,  to \eqref{mild} is given by
\begin{equation}
\psi_t[g](r,\upsilon) : = \mathbb{E}_{\delta_{(r, \upsilon)}}[\langle g, X_t \rangle], \qquad t\geq 0, r\in \bar{D}, \upsilon\in{V},
\label{semigroup}
\end{equation}
for $ g\in L^+_\infty(D\times V)$.
The NBP is thus parameterised by the quantities $\sigma_{\texttt s}, \pi_{\texttt s}, \sigma_{\texttt f}$ and the family of measures ${\mathcal P} =({\mathcal P}_{(r,\upsilon)} , r\in D,\upsilon\in V)$ and accordingly we refer to it as a $(\sigma_{\texttt s}, \pi_{\texttt s}, \sigma_{\texttt f}, \mathcal{P})$-NBP. It  is associated to the NTE via the relation \eqref{semigroup}, but this  association does not uniquely identify the NBP.  Nonetheless for a given quadruple $(\sigma_{\texttt s}, \pi_{\texttt s}, \sigma_{\texttt f}, \pi_{\texttt f})$, it was shown in \cite{SNTE} that under the assumptions (H1) and (H3), at least one NBP always exists that can be associated to it via \eqref{semigroup}.
\smallskip

There is, however, a second equation similar to \eqref{mild}, which describes the non-linear semigroup of the neutron branching process and which {\it does} uniquely identify the $(\sigma_{\texttt s}, \pi_{\texttt s}, \sigma_{\texttt f}, \mathcal{P})$-NBP.  Write the {\it branching generator} associated with the physical process by\footnote{Here and elsewhere, an empty product is always understood to be unity.}
\begin{equation}
G[g](r,\upsilon) =  \sigma_{\texttt f}(r,\upsilon){\mathcal E}_{(r,\upsilon)}\Bigg[\prod_{j =1}^N g(r, \upsilon_j) - g(r,\upsilon)\Bigg]
\label{G}
\end{equation}
for $r\in D$, $\upsilon \in V$ and $g\in L^+_\infty(D\times V)$ and define
\begin{equation}
u_t[g](r,\upsilon): = \mathbb{E}_{\delta_{(r, \upsilon)}}\left[\prod_{i=1}^{N_t}g(r_i(t), \upsilon_i(t)) \right], \qquad t\geq 0.
\label{ut}
\end{equation}
{\color{black} Formally speaking, by extending the domain in which particles live to include a cemetery state $\dagger$, corresponding to neutron capture or  neutrons going to the boundary $\partial D$,  we will always work with the convention (cf. \cite{INW1, INW2, INW3}) that functions appearing in additive functionals are valued as zero on $\{\dagger\}$, whereas in multiplicative functionals, they are valued as one on $\{\dagger\}$. One may think of this as requiring that empty sums are valued as zero where as empty products are valued as one. }
\smallskip

{\color{black}As shown in Section 8 of~\cite{SNTE}, we can break the expectation over the event of scattering or fission in \eqref{ut} and, appealing to standard manipulations (cf. \cite{MultiNTE, SNTE}) we see that, for $g\in L^+_\infty(D\times V)$, which is uniformly bounded by unity,
\begin{equation}
u_t[g] = \hat{\U}_t[g] +\int_0^t \U_s[\bS u_{t-s}[g] + G[u_{t-s}[g]]\d s, \qquad t\geq 0,
\label{non-linear}
\end{equation}
where
\begin{equation}
\hat{\U}_t[g](r, \upsilon) = g(r + \upsilon(t \wedge \kappa_{r, \upsilon}^D), \upsilon).
\label{advhat}
\end{equation}
Under the assumptions (H1), (H2) and (H4), it was also shown in \cite{SNTE} that \eqref{non-linear} has a unique solution in the space of non-negative functions, which are bounded over bounded intervals of time.}
\smallskip

Before moving on to the asymptotics of $(\psi_t, t\geq 0)$, let us make an important note regarding alternative representations of equations \eqref{mild} and \eqref{non-linear} for later use. In order to do so, let us momentarily introduce what we mean by a neutron random walk (NRW); cf. \cite{SNTE}.  A NRW on $D$,  is defined by its scatter rates, $\alpha(r,\upsilon)$, $r\in D, \upsilon\in V$, and scatter probability densities $\pi(r,\upsilon,\upsilon')$, $r\in D, \upsilon,\upsilon'\in V$ where $\textstyle{\int_V \pi(r,\upsilon, \upsilon')\d\upsilon'}=1$ for all $r\in D, \upsilon\in V$. When issued from $r \in D$ with a velocity $\upsilon$, the NRW will propagate linearly with that velocity until either it exits the domain $D$, in which case it is killed, or at the random time $T_{\texttt{s}}$ a scattering occurs, where
$
\Pr(T_{\texttt{s}}>t) = \exp\{-\textstyle{\int_0^t} \alpha(r+\upsilon s, \upsilon){\rm d}s \}, $ for $t\geq0.$ 
When the scattering event occurs in  position-velocity configuration $(r,\upsilon)$, a new velocity $\upsilon'$ is selected with probability $\pi(r,\upsilon,\upsilon')\d\upsilon'$. We refer more specifically to the latter as an $\alpha\pi$-NRW.
\smallskip

The linear mild equation \eqref{mild} and its accompanying non-linear mild form \eqref{non-linear}, although consistent with existing literature (cf. \cite{MultiNTE, SNTE, SNTE-III, MCNTE}) can be equally identified as the unique (in the same sense as mentioned in the previous paragraph) solution to the equations
\begin{equation}
\label{Qlin}
\psi_t[g](r,\upsilon) = {\texttt{Q}}_t[g](r,\upsilon) + \int_0^t {\texttt{Q}}_s[{\bF}\psi_{t-s}[g]](r,\upsilon)\d s, \qquad t\geq 0, r\in D, \upsilon\in V.\end{equation}
and
{\color{black}\begin{equation}
\label{Qnonlin}
u_t[g] = \hat{\Q}_t[g](r,\upsilon) +\int_0^t \Q_s[G[u_{t-s}[g]](r,\upsilon)\d s, \qquad t\geq 0, r\in D, \upsilon\in V,
\end{equation}
respectively, where for $g\in L_\infty^+(D\times V)$,
\[
\Q_t[g](r,\upsilon) = \mathbf{E}_{(r,\upsilon)}[g(R_t,\Upsilon_t)\mathbf{1}_{(t<\tau^D)}],
\]
and 
\[
\hat{\Q}_t[g](r,\upsilon) = \mathbf{E}_{(r,\upsilon)}[g(R_{t\wedge \tau^D},\Upsilon_{t\wedge \tau^D})],
\]
are the expectation semigroups associated with the $\sigma_{\texttt s}\pi_{\texttt s}$-NRW and $\tau^D = \inf\{t>0: R_t\not\in D\}$.}
%where, $(\Q_t, t\geq 0)$ is the expectation semigroup associated with the $\sigma_{\texttt s}\pi_{\texttt s}$-NRW. More precisely, if we write $((R_t,\Upsilon_t), t\geq 0)$ for the pathwise trajectory of the latter, and denote by $\mathbf{P}_{(r,\upsilon)}$, $r\in D$, $\upsilon\in V$, the associated Markov probabilities, then, for $g\in L_\infty^+(D\times V)$, 
%\[
%\Q_t[g](r,\upsilon) = \mathbf{E}_{(r,\upsilon)}[f(R_t,\Upsilon_t)\mathbf{1}_{(t<\tau^D)}],
%\]
%where $\tau^D = \inf\{t>0: R_t\not\in D\}$.
%\smallskip

\subsection{Lead order asymptotics of the expectation semigroup}

In the accompanying predecessor to this article, \cite{SNTE}, a Perron-Frobenius type asymptotic was developed for $(\psi_t, t\geq 0)$. In order to state it we need to introduce another assumption, which is slightly stronger than (H2). To this end, define 
\begin{align}
\alpha(r,\upsilon)\pi(r, \upsilon, \upsilon') = \sigma_{\texttt{s}}(r,\upsilon)\pi_{\texttt{s}}(r, \upsilon, \upsilon') + \sigma_{\texttt{f}}(r,\upsilon) \pi_{\texttt{f}}(r, \upsilon, \upsilon')\qquad r\in D, \upsilon, \upsilon'\in V.
\label{alpha}
\end{align}
Our new condition is:
\smallskip

{\color{black}
{\bf (H2)$^*$:  We have $\textstyle{\inf_{r\in D, \upsilon, \upsilon'\in V} \alpha(r,\upsilon)\pi(r,\upsilon,\upsilon')>0}$.}
}

\begin{theorem}\label{PF}
Suppose that (H1) and (H2)$^*$ hold.
Then, for semigroup  $(\psi_t,t\geq0)$  identified by \eqref{mild},  there exists  a $\lambda_*\in\mathbb{R}$, a  positive\footnote{To be precise, by a positive eigenfunction, we mean a mapping from $D\times V\to (0,\infty)$. This does not prevent it being valued zero on $\partial D$, as $D$ is an open bounded, convex domain.} right eigenfunction $\varphi \in L^+_\infty(D\times V)$ and a left eigenmeasure which is absolutely continuous with respect to Lebesgue measure on $D\times V$ with density $\tilde\varphi\in L^+_\infty(D\times V)$, both having associated eigenvalue ${\rm e}^{\lambda_* t}$, and such that $\varphi$  (resp. $\tilde\varphi$) is uniformly (resp. a.e. uniformly) bounded away from zero on each compactly embedded subset of $D\times V$. In particular for all $g\in L^+_{\infty}(D\times V)$ 
\begin{equation}
\langle\tilde\varphi, \psi_t[g]\rangle = {\rm e}^{\lambda_* t}\langle\tilde\varphi,g\rangle\quad  \text{(resp. } 
\psi_t[\varphi] = {\rm e}^{\lambda_* t}\varphi
\text{)} \quad t\ge 0.
\label{leftandright}
\end{equation}
Moreover, there exists $\varepsilon>0$ such that, for all $g\in L^+_\infty(D\times V)$,
\begin{equation}
  \left\|{\rm e}^{-\lambda_* t}{\varphi}^{-1}{\psi_t[g]}-\langle\tilde\varphi, g\rangle\right\|_\infty = O({\rm e}^{-\varepsilon t}) \text{ as $t\rightarrow+\infty$.}
\label{phiasymp}
\end{equation}

\end{theorem}

\smallskip

In light of Theorem \ref{PF}, we can categorise the physical process according to the value of $\lambda_*$. In particular, when $\lambda_*>0$ we say the process is supercritical, when $\lambda_* =0$, the process is critical and when $\lambda_*<0$, the process is subcritical.

\subsection{Strong law of large numbers at supercriticality}
The main aim of this article as a continuation of \cite{SNTE} is to understand the almost sure behaviour of the $(\sigma_{\texttt s}, \pi_{\texttt s}, \sigma_{\texttt f}, \mathcal{P})$-NBP in relation to what is, in effect, a statement of mean growth in Theorem \ref{PF}, in the setting that $\lambda_*>0$. In the aforesaid article, it was noted that 
\begin{equation}
W_t: = {\rm e}^{-\lambda_* t}\frac{\langle \varphi, X_t\rangle}{\langle\varphi, \mu\rangle}, \qquad t\geq 0,
\label{addmg}
\end{equation}
is a unit mean  martingale under $\mathbb{P}_\mu$, $\mu\in \mathcal{M}(D\times V)$ and, moreover its convergence was studied. 
{\color{black}
In particular, since the martingale is non-negative, we automatically know that it must converge to a limiting random variable, that is, ${W_t\rightarrow W_\infty}$, $\mathbb{P_\mu}$-almost surely, where we can take $\textstyle{W_\infty:=\liminf_{t\to0} W_t}$ for definiteness. 
}
Before stating the result regarding the latter, we require one more assumption on the NBP:
\bigskip

{\color{black}\bf  (H3)$^*$: There exists a ball $B$ compactly embedded in $D$ such that 
\[
\inf_{r \in B, \upsilon, \upsilon' \in V}\sigma_{\texttt{f}}(r, \upsilon)\pi_{\texttt{f}}(r, \upsilon, \upsilon') > 0.
\]
}

The following result was derived in \cite{SNTE}.
\begin{theorem}\label{Kesten} For the  $(\sigma_{\emph{\texttt s}}, \pi_{\emph{\texttt s}}, \sigma_{\emph{\texttt f}}, \mathcal{P})$-NBP satisfying the assumptions (H1), (H2)$^*$, (H3)$^*$ and (H4), the martingale $(W_t, t\geq 0)$ 
%is $L_2(\mathbb{P})$ convergent 
{\color{black}
converges to $W_\infty$ $\mathbb{P}$-almost surely and in $L_2(\mathbb{P})$ 
if and only if $\lambda_*>0$,
%(in which case it is also $L_1(\mathbb{P})$ convergent),
otherwise $W_\infty=0$ $\mathbb{P}$-almost surely. 
%when $\lambda_*\leq 0$. 
}
\end{theorem}

Note that when $\lambda_*\leq 0$, since $\textstyle{\lim_{t\to0}W_t =0}$ almost surely, it follows that, for each $\Omega$ compactly embedded in $D\times V$, $\textstyle{\lim_{t\to\infty}X_t(\Omega)=0}$. It therefore remains to describe the growth of $X_t(\Omega)$, $t\geq0$, for $\lambda_* > 0$. This is the main result of this paper, given below.  In order to state it, we must introduce the notion of a directionally continuous function on $D\times V$. Such functions are defined as having the property that, for all $r\in D$, $\upsilon \in V$, 
\[
\lim_{t\downarrow0} g(r+\upsilon t, \upsilon) = g(r,\upsilon).
\]

\begin{theorem}\label{SLLN} 
Suppose the assumptions of Theorem \ref{Kesten} hold. 
For all measurable and directionally continuous 
{\color{black} non-negative} 
$g$  on $D\times V$
such that, up to a multiplicative constant, $g\leq \varphi$, 
%under the assumptions of Theorem \ref{Kesten},
then for any  initial configuration $\mu\in\mathcal{M}(D\times V)$,
\[
%\lim_{t\to\infty} 
{\rm e}^{-\lambda_* t}\frac{\langle g, X_t\rangle}{\langle\varphi, \mu\rangle}
%=  
\rightarrow \langle g,\tilde{\varphi}\rangle W_\infty 
%\quad \text{ as } t\rightarrow \infty
\]
$\mathbb{P}_\mu$-almost surely {\color{black} and in $L_2(\mathbb{P})$}, as $t\rightarrow\infty$.
\end{theorem}

To the best of our knowledge no such results can be found in the existing neutron transport literature. The closest known results are found in the final section of \cite{MWY} and are significantly weaker than Theorem \ref{SLLN}.

\smallskip

We can think of Theorem \ref{SLLN} as stating a stochastic analogue of \eqref{phiasymp},
{\color{black}
noting, for example, that the former implies 
\begin{equation}
\lim_{t\to\infty}{\rm e}^{-\lambda_*t}\frac{\mathbb{E}_{\delta_{(r, \upsilon)}}[\langle g, X_t \rangle]}{\varphi(r,v)} = \langle g,\tilde\varphi\rangle 
%\mathbb{E}_{\delta_{(r, \upsilon)}} W_\infty
\label{allthemass}
\end{equation}
for all $r\in D$, $\upsilon \in V$, which is a version of the latter (albeit without the speed of convergence).
}
%since the former states that, for  $\mu\in\mathcal{M}(D\times V)$,
%\begin{equation}
%\lim_{t\to\infty}{\rm e}^{-\lambda_*t}\frac{\langle\varphi g, X_t\rangle}{\langle \varphi,\mu\rangle} = \langle\tilde\varphi\varphi, g\rangle W_\infty.
%\label{allthemass}
%\end{equation}
%$\mathbb{P}_{\mu}$-almost surely.
\smallskip

The proof of  Theorem \ref{SLLN} relies on a fundamental path decomposition, often referred to in the theory of spatial and non-spatial branching processes as a {\it skeletal decomposition}, see e.g. \cite{EKW, HHK, BKM, DW07, MSP}. The skeletal decomposition is essential in that it identifies an embedded NBP within the original one for which there is no neutron-absorption (neither at $\partial D$ nor into nuclei at collision). This `thinned down tree' is significantly easier to analyse for technical reasons, but nonetheless provides all the mass in the limit \eqref{allthemass}.

\section{Skeletal decomposition}\label{Sk} Inspired by \cite{HHK}, we dedicate this section to the proof of a so-called skeletal decomposition, which necessarily requires us to have $\lambda_*>0$. In very rough terms, for the NBP, we can speak of genealogical lines of descent, meaning neutrons that came from a fission event of a neutron that came from a fission event of a neutron ... and so on, back to one of the initial neutrons a time $t=0$. If we focus on an individual  genealogical line of descent embedded in the NBP, it has a space-velocity trajectory which takes the form of a NRW whose spatial component may or may not hit the boundary of $D$. Indeed, when the NBP survives for all time (requiring $\lambda_*>0$), there must necessarily be some genealogical lines of descent whose spatial trajectories remain in $D$ forever. \smallskip

The basic idea of the skeletal decomposition is to consider the collection of all surviving genealogical lines of descent and understand their space-velocity dynamics collectively as a process ({\it the skeleton}). It turns out that the skeleton forms another NBP but with different scatter and fission statistics from the underlying NBP, due to the fact that we are considering genealogical lines of descent which are biased, since they remain in $D$ for all time. 
For the remaining neutron trajectories that go to the boundary of $D$ or end in neutron capture,  the skeletal decomposition identifies them as immigrants that are thrown off the path of the skeleton.
\smallskip

Below, we develop the statement of the skeletal decomposition. It was brought to our attention by a referee that the proof is robust enough to work in the relatively general setting of a {\it Markov branching process (MBP) with non-local branching} and hence we first set up the notation of a general branching process. It is worthy of note that the motivation for this switch to a general setting is that, for branching particle systems, nothing is known of skeletal decompositions for non-local branching generators; although some results have been identified in the more continuous setting of superprocesses, cf \cite{MSP}, they do not apply to particle systems. Our proof is  inspired by the martingale arguments found in \cite{HHK} which gives a skeletal decomposition for branching Brownian motion in a strip with local branching.

\subsection{The general branching Markov setup}\label{MBP}
Until the end of this section (Section \ref{Sk}), unless otherwise mentioned, we will work in the setting of a general MBP, which we will shortly define in more detail. The reader will note that we necessarily choose to overlap our notation for this general setting with that of the NBP. As such, the reader is encouraged to keep in mind the application to the NBP at all times. Additionally, we provide some remarks at the end of this section to illustrate how the general case takes a specific form in the case of the NBP.
\smallskip

%
%   COMMENT
%
%\marginpar{\scriptsize\color{red}  NOTE: Consider $(x_1,\cdots, x_N)$ as ordered. If an unordered set $\{x_1,\cdots, x_N\}$ then would need to be careful of extra combinatorial factors etc. Ordering offspring would also lends itself to Ulam-Harris labelling.}
%
%
%
Henceforth, $X =(X_t, t\geq0)$ will be a $(\bP, G)$-Markov branching process on a non-empty, open Euclidian domain\footnote{The arguments presented here are robust enough to work with more abstract domains; see for example the set up in \cite{AH}.} $E \subseteq\mathbb{R}^d$, where $\bP = (\bP_t, t\geq0)$ is a Markov semigroup on $E$ and $G$ is the associated branching generator. More precisely, $X$ is an atomic measure-valued stochastic process (in a similar sense to \eqref{atomicvalued}) in which particles move independently according to a copy of the Markov process associated to $\bP$ such that, when a particle is positioned at $x\in E$,  at the instantaneous spatial rate $\varsigma(x)$, the process will branch and a random number of  offspring, say $N$, are thrown out in positions, say $x_1,\cdots, x_N$ in $E$, according to some law $\mathcal{P}_x$.
{\color{black} (Note, we always consider of $(x_1,\dots,x_N)$ as an ordered set of points.)}
\smallskip

We do not need $\bP$ to have the Feller property, and we assume nothing of the boundary conditions on $E$, in particular, $\bP$ need not be conservative. That said, it will prove to be more convenient to introduce a (possible) cemetery state $\dagger$ appended to $E$, which is to be treated as an absorbing state, and regard $\bP$ as conservative. 
{\color{black} 
As such, 

\begin{equation}
\bP_t[f](x)=\mathbf{E}_x[f(\xi_t)] = \mathbf{E}_x[f(\xi_t)\mathbf{1}_{(t<\tk)}] ,\qquad x\in E, f\in L_\infty^{+}(E),
\label{PtoP}
\end{equation}
where  the process $\xi$, with probabilities $(\mathbf{P}_x,x\in E)$, is the Markov process on $E\cup\{\dagger\}$ with lifetime $\tk =\inf\{t>0: \xi_t \in\{\dagger\}\}$,  $L_\infty^{+}(E)$ is the space of bounded, measurable functions on $E$ and, in this context, we always take $f(\dagger): = 0$. 
}
\smallskip

As such, 
  in a similar spirit to \eqref{G}, we can think of  the branching generator, $G$, as having definition
\begin{equation}
G[f](x) =  \varsigma(x){\mathcal E}_x\Bigg[\prod_{j =1}^N f(x_j) - f(x)\Bigg], \qquad x\in E,
\label{similar111}
\end{equation}
 for $f\in L_\infty^{+,1}(E)$, the space of non-negative   measurable functions on $E$ bounded by unity. {\color{black} As previously, we always define the empty product as equal to unity.}
\smallskip

We  use $\P_{\delta_x}$ for the law of $X$  issued from a single particle positioned at $x\in E$.
{\color{black}  In a similar spirit to
\eqref{ut}, we can introduce the non-linear semigroup of the branching process,
 \begin{equation}
 u_t[g](x): =\mathbb{E}_{\delta_x}\left[\prod_{i = 1}^{N_t} g(x_i(t))\right], \qquad t\geq 0, x\in E, g\in L_\infty^{+,1}(E),
 \label{MBPnonlinear}
 \end{equation}
where $X_t = \textstyle{\sum_{i = 1}^{N_t}\delta_{x_i(t)}}$, $t\geq 0$. As before, we define the empty product to be unity, and for consistency, functions, $g$, appearing in such functionals can be valued on $E\cup\{\dagger\}$ and forced to take the value $g(\dagger) = 1$.
 \smallskip
 
Similarly to the derivation of \eqref{non-linear} and \eqref{Qnonlin}, it is straightforward to show that, for such functions, $u_t[g]$
%for $g\in L_\infty^{+,1}(E)$ with $g(\dagger) = 1$, 
%  the semigroup of the branching process,
% \begin{equation}
% u_t[g](x): =\mathbb{E}_{\delta_x}\left[\prod_{i = 1}^{N_t} g(x_i(t))\right], \qquad t\geq 0, x\in E,
% \label{MBPnonlinear}
% \end{equation}
solves the non-linear mild equation
\begin{equation}
u_t[g] = \hat\bP_t[g](x) +\int_0^t \bP_s[G[u_{t-s}[g]](x)\d s, \qquad t\geq 0, x\in E,
\label{genut}
\end{equation}
where we need to adjust $\bP$ to $\hat\bP$ to accommodate for the fact that empty products are valued as one, as follows
\begin{equation}
\hat\bP_t[g](x)= \mathbf{E}_x[g(\xi_{t\wedge\tk})] ,\qquad x\in E.
\label{hatP}
\end{equation}

}
\smallskip

Now, define 
\begin{equation}
\zeta \coloneqq \inf\{t \ge 0 : \langle 1,X_t \rangle= 0\},
\label{zeta}
\end{equation}  
the time of extinction, and let 
\begin{equation}
\label{wdef}
w(x) \coloneqq \P_{\delta_{x}}(\zeta < \infty).
\end{equation} 
We will also frequently  use with 
\[
p(x): = 1-w(x), \qquad x\in E.
\]
 {\color{black} Recalling that  we need to take as a definition $w(\dagger) =1$,  by conditioning on $\mathcal{F}_t = \sigma(X_s, s\leq t)$, for $t\geq 0$,
\begin{equation}
w(x) = \E_{\delta_{x}}\left[\prod_{i =1}^{N_t} w(x_i(t))\right].
\label{fixedpoint}
\end{equation}
Taking \eqref{fixedpoint}, \eqref{genut} and \eqref{hatP} into account, it is easy to deduce that $w$ also solves 
\begin{equation}
w(x) = \hat\bP_t[w](x) + \int_0^t\bP_s\left[ G[w]\right](x) \D s, \qquad t\geq 0, x\in E.
\label{nonlinear}
\end{equation}
}
\smallskip

We will  assume: 
\bigskip

{\bf (M1): $\textstyle{\inf_{x\in E}w(x)>0}$ and  $w(x)<1$  for $x\in E$.}

\smallskip

Beyond this,  we assume relatively little about $\bP$ and $G$ other than:
\bigskip

{\bf (M2): The branching rate $\varsigma$ is uniformly bounded from above.}

% and by extending the domain of all functions henceforth mentioned, we value them as zero on $\{\dagger\}$.
 \smallskip

%Using the definition of $\texttt{U}_t$ from~\eqref{adv}, after a little rearrangement,~\eqref{nonlinear} becomes
%\begin{align}
%w(r, \upsilon) &=w(r+\upsilon t, \upsilon)\mathbf{1}_{(t<\kappa^D_{r,\upsilon})} + \int_0^{t\wedge \kappa^D_{r,\upsilon}}\left(\frac{\bS{w}(r+\upsilon s , \upsilon)}{w(r+\upsilon s , \upsilon)}  + \sigma_\texttt{s}(r+\upsilon s, \upsilon)\right)w(r+\upsilon s , \upsilon)\D s\notag \\
%&\quad + \int_0^{t\wedge \kappa^D_{r,\upsilon}} \frac{G[w](r+\upsilon s, \upsilon)}{w(r+\upsilon s, \upsilon)}w(r+\upsilon s, \upsilon)\D s - \int_0^{t\wedge \kappa^D_{r,\upsilon}}\sigma_\texttt{s}(r+\upsilon s, \upsilon) w(r+\upsilon s, \upsilon)\D s.
%\label{preexp}
%\end{align}
Re-writing \eqref{nonlinear} in the form
\[{\color{black}
w(x)= \mathbf{E}_x[w(\xi_{t\wedge\tk})] + \mathbf{E}_x\left[\int_0^{t\wedge \tk}  w(\xi_s) \frac{G[w](\xi_s)}{w(\xi_s)}\D s\right],\qquad  t\geq 0,}
\]
and noting that $\textstyle{\sup_{x\in E}G[w](x)/w(x)<\infty}$,
we can appeal to the method of exchanging exponential potential for additive potential\footnote{We will use this trick throughout this paper and  consistently refer to it as the `transfer of the exponential potential to the additive potential' and vice versa in the other direction.}  in e.g.~\cite[Lemma 1.2, Chapter 4, Part 1]{Dynkin2}, which   yields
%\begin{align}
%\mathbf{1}_{(t<\kappa^D_{r,\upsilon})} \frac{w(r+\upsilon t, \upsilon)}{w(r, \upsilon)} &={\rm exp}\left(-\int_0^{t\wedge \kappa^D_{r,\upsilon}}  \frac{G[w](r+\upsilon s, \upsilon)}{w(r + \upsilon s, \upsilon)} + \frac{\bS{w}(r+\upsilon s , \upsilon)}{w(r+\upsilon s , \upsilon)} \D s \right).
%\label{exp}
%\end{align}
\begin{equation}
w(x) = \mathbf{E}_x\left[w(\xi_{t\wedge\tk}) \exp\left(\int_0^{t\wedge\tk}\frac{G[w](\xi_s)}{w(\xi_s)}\d s\right)\right], \qquad x\in E, t\ge 0.
\label{exp}
\end{equation}

 This identity will turn out to be extremely useful in our analysis, in particular, the equality \eqref{exp} together with the Markov property of $\xi$ implies that the object in the expectation on the right-hand side of \eqref{exp} is a martingale.
\smallskip

In Theorem \ref{skeleton} below  we give the skeletal decomposition in the form of a theorem. In order to state this result, we first need to develop two notions of conditioning. As there is rather a lot of notation, we include a table in the Appendix which the reader may refer to as needed.

\smallskip

The basic pretext of the skeletal decomposition is that we want to split genealogical lines of descent into those that survive forever and those that reach a dead end. To this end, let $c_i(t)$ denote the label of a particle $i \in \{1, \dots, N_t\}$. We label a particle `prolific', denoted $c_i(t) =\, \uparrow$, if it has an infinite genealogical line of descent, and  $c_i(t) =\, \downarrow$, if its line of descent dies out (i.e. `non-prolific'). Ultimately, we want to describe how the spatial genealogical tree of the MBP can be split into  into a spatial genealogical sub-tree, consisting of $\uparrow$-labelled particles (the skeleton), which is dressed with trees of $\downarrow$-labelled particles.
\smallskip

Let $\P^\updownarrow = (\P^\updownarrow_{\delta_{x}}, x\in E)$ denote the probabilities of the two-labelled process described above. Then for $t\ge 0$ and $x \in E$ we have the following relationship between $\P^\updownarrow$ and $\P$:
\begin{equation}
\frac{\D\P^\updownarrow_{\delta_{x}}}{\D\P_{\delta_{x}}}\bigg|_{\mathcal{F}_{\infty}} = \prod_{i = 1}^{N_t}\left(\mathbf{1}_{(c_i(t) = \,\uparrow)} + \mathbf{1}_{(c_i(t) = \,\downarrow)} \right) = 1,
\end{equation}
where $\mathcal{F}_\infty = \sigma\left(\cup_{t\geq 0} \mathcal{F}_t\right)$.
Projecting onto $\mathcal{F}_t$, for $t\geq 0$, we have
\begin{align}
\frac{\D\P^\updownarrow_{\delta_{x}}}{\D\P_{\delta_{x}}}\bigg|_{\mathcal{F}_{t}} &= \E_{\delta_{x}}\left(\prod_{i = 1}^{N_t}\left(\mathbf{1}_{(c_i(t) =\,\uparrow)} + \mathbf{1}_{(c_i(t) =\,\downarrow)} \right) \bigg| \mathcal{F}_t\right)\notag\\
&= \sum_{I \subseteq \{1,\dots N_t\}}\prod_{i\in I}\P_{\delta_{x}}(c_i(t) =\,\uparrow |\mathcal{F}_t)\prod_{i\in \{1,\dots, N_t\}\backslash I}\P_{\delta_{x}}(c_i(t) =\,\downarrow |\mathcal{F}_t)\notag \\
&= \sum_{I \subseteq \{1,\dots N_t\}}\prod_{i\in I} p(x_i(t))\prod_{i\in \{1,\dots, N_t\}\backslash I}w(x_i(t)),
\label{updownCOM}
\end{align}
where we understand the sum to be taken over all subsets  of $\{1, \cdots, N_t\}$, each of which is denoted by $I$.
\smallskip

The decomposition in  \eqref{updownCOM} indicates the beginning point of how we break up the law of the $(\bP, G)$-MBP according to subtrees that are categorised as $\downarrow$ (with probability $w$) and subtrees that are categorised as $\uparrow$ with $\downarrow$ dressing (with probability $p$), the so-called {\it skeletal decomposition}.
\smallskip

In the next two sections we will examine the notion of our MBP conditioned to die out and conditioned to survive, respectively. Thereafter we will use the characterisation of these conditioned trees to formalise our skeletal decomposition. 

%%%%%%%%%%%%%%
%%%%%%%%%%%%%%
%%%%%%%%%%%%%%
%%%%%%%%%%%%%%
%%%%%%%%%%%%%%
%%%%%%%%%%%%%%
%% DOWNARROW TREES%%
%%%%%%%%%%%%%%
%%%%%%%%%%%%%%
%%%%%%%%%%%%%%
%%%%%%%%%%%%%%
%%%%%%%%%%%%%%
%%%%%%%%%%%%%%
%%%%%%%%%%%%%%
%%%%%%%%%%%%%%

\subsection{$\downarrow$-trees}\label{parti}
Following~\cite{HHK}, let us start by characterising the law of genealogical trees populated by the marks $\downarrow$. Thanks to the branching property, it suffices to consider trees which are issued with a single particle with mark  $\downarrow$. By definition of the mark $c_\emptyset(0)  =\,\downarrow$, where $\emptyset$ is the initial ancestral particle, this is the same as understanding the law of $(X,\P)$  conditioned to become extinct. Indeed, for $A \in \mathcal{F}_t$,
\begin{align}
\P_{\delta_{x}}^\downarrow(A) &\coloneqq \P^\updownarrow_{\delta_{x}}(A | c_\emptyset(0) =\,\downarrow)\notag\\
&= \frac{\P^\updownarrow_{\delta_{x}}(A ; c_i =\,\downarrow, \text{ for each }i = 1, \dots, N_t)}{\P^\updownarrow_{\delta_{x}}(c_\emptyset(0) =\,\downarrow)}\notag\\
&= \frac{\E_{\delta_{x}}\left[\mathbf{1}_A\prod_{i=1}^{N_t}w(x_i(t))\right]}{w(x)}.\label{probredtree}
\end{align}
We are now in a position to characterise the MBP trees which are conditioned to become extinct (equivalently, with genealogical lines of descent which are marked entirely with $\downarrow$). Heuristically speaking, the next proposition shows that the conditioning creates a branching particle  process in which particles are prone to die out (whether that be due to being killed at the boundary under $\bP$, or suppressing offspring). Our proof is partly inspired by Proposition 11 of \cite{HHK}.

\begin{prop}[$\downarrow$ Trees]\label{black}
For initial configurations of the form $\textstyle{\nu = \sum_{i=1}^n\delta_{x_i}} $, for $n\in\mathbb{N}$ and $x_1,\cdots, x_n\in E$, define the measure $\mathbb{P}_\nu^\downarrow$ via ~\eqref{probredtree},  
\[
\mathbb{P}_\nu^\downarrow = \otimes_{i = 1}^n\P^\downarrow_{\delta_{x_i}},
\] 
i.e.  starting independent processes at positions $x_i$ each under $\mathbb{P}_{\delta_{x_i}}$, for $i =1,\cdots, n$.
Then under $\mathbb{P}_\nu^\downarrow$, $X$ is a $(\ebP^\downarrow, G^\downarrow)$-MBP with motions semigroup $\ebP^\downarrow$ and branching generator $G^\downarrow$ defined as follows. 
The motion semigroup $\ebP^\downarrow$ is that of the Markov process $\xi$ with probabilities $(\mathbf{P}^\downarrow_x, x\in E)$, where 
\begin{equation}
\left.\frac{\d \mathbf{P}^\downarrow_{x}}{\d \mathbf{P}_{x}}\right|_{\sigma(\xi_s, s\leq t)} = \frac{w(\xi_{t\wedge \etk})}{w(x)}\exp\left(\int_0^{t\wedge \etk} \frac{G[w](\xi_s)}{w(\xi_s)}\d s\right), \qquad t\geq 0.
\label{COMdown}
\end{equation}
For $x\in E$ and $f\in L_\infty^{+,1}(E)$, the branching generator is given by 
\begin{equation}
G^\downarrow[f]  = \frac{1}{w}\left[G[fw] - fG[w] \right],\label{redbmech}
\end{equation}
which may otherwise be identified as 
\[
G^\downarrow[f] = \varsigma^\downarrow(x)\mathcal{E}^\downarrow_{x}\Bigg[\prod_{j =1}^N f(x_j) - f(x)\Bigg],
\]
where
\begin{equation}
\varsigma^\downarrow(x) = \varsigma(x) + \frac{G[w](x)}{w(x)}
= \frac{\varsigma(x)}{w(x)}\mathcal{E}_{x}\Bigg[\prod_{j =1}^N w(x_j)\Bigg]
\qquad x\in E,
\label{redbrate}
\end{equation}
and
\begin{equation}
\left.\frac{\d \mathcal{P}^\downarrow_{x}}{\d \mathcal{P}_{x}}\right|_{\sigma(N, x_1, \dots, x_N)} 
= \frac{\prod_{i=1}^Nw(x_i)}{\mathcal{E}_{x}\left[\prod_{j =1}^N w(x_j)\right]}
= \frac{\varsigma(x)}{\varsigma^\downarrow(x)w(x)}\prod_{i=1}^Nw(x_i)
\label{COMcalP}
\end{equation}
%where %$\mathcal{X} = \sigma(x_i, i = 1, \cdots, N)$ 
%$\mathcal{X} = \sigma(N, x_1, \dots, x_N)$.
%More precisely, if we write 
%\begin{equation}
%u^\downarrow_t[g](x): = \E^\updownarrow_{\delta_{x}}\left[\left.\prod_{i = 1}^{N_t} g(x_i(t))\right| c_\emptyset(0) =\,\downarrow\right] = \frac{1}{w(x)}u_t[w g](x),
%\label{udown}
%\end{equation}
%then, in the spirit of \eqref{non-linear}, we have, for $g\in L_\infty^{+,1}(E)$, 
%\begin{equation}
%u^\downarrow_t[g](x) = \hat\ebP^\downarrow_t[g](x) +\int_0^t \ebP^\downarrow_s[G^\downarrow[u^\downarrow_{t-s}[g]](x)\d s, \qquad t\geq 0, x\in E.
%\label{rednonlinearNTE}
%\end{equation}
%where $\hat\ebP^\downarrow$ is defined in a similar spirit to \eqref{hatP}.
\end{prop}

\begin{proof}[Proof of Proposition \ref{black}] First let us show that the change of measure results in a particle process that respects the Markov branching property.
In a more general sense, for $\nu$ as in the statement of this proposition, \eqref{probredtree} takes the form 
\[
\left.\frac{\d \mathbb{P}^\downarrow_\nu}{\d \mathbb{P}_\nu}\right|_{\mathcal{F}_t} = \frac{\prod_{i = 1}^{N_t} w(x_i(t))}{\prod_{i=1}^n w(x_i)}
\]
It is clear from the conditioning that every particle in the resulting process under the new  measure $\mathbb{P}^\downarrow_\nu$ must carry the mark $\downarrow$, i.e. be non-prolific, by construction. 

\smallskip

Let us define, for $g\in L_\infty^{+,1}(E)$, 
\begin{equation}
u^\downarrow_t[g](x) = \E^\updownarrow_{\delta_{x}}\left[\left.\prod_{i = 1}^{N_t} g(x_i(t))\right| c_\emptyset(0) =\,\downarrow\right] = \frac{1}{w(x)}u_t[w g](x),
\label{udown}
\end{equation}
which describes the evolution of the the process $X$ under $\mathbb{P}^\downarrow$. In particular, for $g\in L_\infty^{+,1}(E)$, $x\in E$ and $s,t\geq 0$, note that
\begin{align}
\mathbb{E}^\downarrow_{\delta_{x}}\left[\left.\prod_{i = 1}^{N_{t+s}}g(x_i(t+s))\right|\mathcal{F}_t\right]
& =\frac{1}{w(x)}\prod_{i = 1}^{N_{t}}w(x_i(t))
\mathbb{E}_{\delta_{x}}\left[\left.\frac{\prod_{j = 1}^{N^i_{s}}w(x_j^i(s))g(x_j^i(s))}{w(x_i(t))}\right|\mathcal{F}_t\right]\notag\\
&= \frac{1}{w(x)}\prod_{i = 1}^{N_{t}} w(x_i(t))u^\downarrow_s[g] (x_i(t)) ,
\label{itisNBP}
\end{align}
where, given $\mathcal{F}_t$, $((x_j^i(t)), j = 1,\cdots, {N}^i_s)$ are the physical configurations of particles at time $t+s$ that are descendent from  particle  $i\in N_t$. This ensures the Markov branching property holds.

\smallskip

 It thus suffices for the remainder of the proof to show,  in the spirit of \eqref{non-linear}, that, for $g\in L_\infty^{+,1}(E)$, 
 \begin{equation}
u^\downarrow_t[g](x) = \hat\bP^\downarrow_t[g](x) +\int_0^t \bP^\downarrow_s[G^\downarrow[u^\downarrow_{t-s}[g]](x)\d s, \qquad t\geq 0, x\in E.
\label{rednonlinearNTE}
\end{equation}
holds, where $\hat{\bP}^\downarrow$ is defined in a similar spirit to \eqref{hatP}, which is the semigroup evolution equation for a $(\bP^\downarrow, G^\downarrow)$-MBP, and to identify the internal structure of $G^\downarrow$.
\smallskip

From \eqref{genut} and \eqref{udown} it follows that, for $g\in L_\infty^{+,1}(E)$, 
{\color{black}
\begin{align}
\label{downarrownon-linear}
u^\downarrow_t[g] 
%&= \frac{1}{w}\bP_t[wg] +\int_0^t \frac{1}{w}\bP_s[ G[u_{t-s}[wg]]\d s\notag\\
&= \frac{1}{w}\hat\bP_t[wg] +\int_0^t \frac{1}{w}\bP_s[G[wu^\downarrow_{t-s}[g]]\d s, \qquad t\geq 0,
\end{align}
}

In the spirit of the derivation of \eqref{exp}, we can 
apply~\cite[Lemma 1.2, Chapter 4, Part 1]{Dynkin1} and use \eqref{COMdown} and \eqref{downarrownon-linear} to get, for $x\in E$,
\begin{align*}
u^\downarrow_t[g](x)
&= {\color{black}\frac{1}{w(x)}\hat\bP_t[wg](x) }
+\int_0^t \frac{1}{w(x)}\bP_s\left[w\frac{G[wu^\downarrow_{t-s}[g]]}{w}\right](x)\d s\\
&+\int_0^t \frac{1}{w(x)}\bP_s\left[\frac{G[w]}{w}wu^\downarrow_{t-s}[g]\right](x)\d s
-\int_0^t  \frac{1}{w(x)}\bP_s\left[\frac{G[w]}{w}w u^\downarrow_{t-s}[g]\right](x)\d s\\
&={\color{black}\frac{1}{w(x)}\mathbf{E}_x\left[g(\xi_{t\wedge\tk}) w(\xi_{t\wedge\tk})
{\rm e}^{\int_0^{t\wedge\tk} \frac{G[w](\xi_u)}{w(\xi_u)}\d u}\right]}\\
&+ \frac{1}{w(x)}\mathbf{E}_x\left[
\int_0^{\color{black} t\wedge\tk}
\frac{G[wu^\downarrow_{t-s}[g]](\xi_s)}{w(\xi_s)}w(\xi_s)
{\rm e}^{\int_0^s \frac{G[w](\xi_u)}{w(\xi_u)}\d u}\d s\right]\\
&-  \frac{1}{w(x)}
\mathbf{E}_x\left[\int_0^{\color{black} t\wedge \tk}
\frac{G[w](\xi_s)}{w(\xi_s)} u^\downarrow_{t-s}[g](\xi_s)w(\xi_s)
{\rm e}^{\int_0^s \frac{G[w](\xi_u)}{w(\xi_u)}\d u}
\d s\right]\\
&= {\color{black}\hat\bP^\downarrow_t[g](x)} +\int_0^t \bP^\downarrow_s\left[\frac{G[wu^\downarrow_{t-s}[g]]}{w}\right](x)\d s-\int_0^t  \bP^\downarrow_s\left[\frac{G[w]}{w}u^\downarrow_{t-s}[g]\right](x)\d s\\
&= {\color{black}\hat\bP^\downarrow_t[g](x)}+\int_0^t \bP^\downarrow_s\left[G^\downarrow[u^\downarrow_{t-s}[g]\right](x)\d s
\end{align*}
where   we have used the definition \eqref{redbmech}.

\smallskip

It remains to identify the internal structure of $G^\downarrow $. Taking as a pre-emptive definition $\varsigma^\downarrow := \varsigma + w^{-1}G[w]$, we have, for $f\in L_\infty^{+,1}(E)$,
\begin{align*}
G^\downarrow[f](x) &= \frac{1}{w(x)}\left[G[fw] - fG[w] \right](x)\notag\\
&= \frac{1}{w(x)}\left[\varsigma(x) \mathcal{E}_{x}\left[\prod_{i = 1}^Nf(x_i)w(x_i)\right] - \varsigma(x)f(x)w(x) -  fG[w](x)\right]\notag\\
&= \frac{\varsigma(x)}{w(x)}\mathcal{E}_{x}\left[\prod_{i = 1}^Nf(x_i)w(x_i) \right]- \left(\varsigma(x) + \frac{G[w]}{w}(x) \right)f(x)\notag\\
&= \varsigma^\downarrow(x)\left(\frac{\varsigma(x)}{\varsigma^\downarrow(x)w(x)}\mathcal{E}_{x}\left[\prod_{i=1}^Nw(x_i) f(x_i)\right]- f(x)\right),\notag
\end{align*}
Moreover, recalling the change of measure \eqref{COMcalP},
note that, for $x\in E$,  $\mathcal{P}^\downarrow_{x}$ is a probability measure  on account of the fact that, when we set $f\equiv 1$, recalling again that $\varsigma^\downarrow = \varsigma + w^{-1}G[w]$ as well as the definition of $G$ given in \eqref{similar111},
\[
\mathcal{E}_{x}\left[\frac{\varsigma(x)}{\varsigma^\downarrow(x)w(x)}\prod_{i=1}^Nw(x_i)\right]=
\frac{ G[w](x) + \varsigma(x)w(x)}{\varsigma(x)+ w^{-1}(x)G[w](x)}\frac{1}{w(x)} = 1
\]

as required.
\end{proof}

%%%%%%%%%%%%%%
%%%%%%%%%%%%%%
%%%%%%%%%%%%%%
%%%%%%%%%%%%%%
%%%%%%%%%%%%%%
%%%%%%%%%%%%%%
%% UPARROW TREES%%
%%%%%%%%%%%%%%
%%%%%%%%%%%%%%
%%%%%%%%%%%%%%
%%%%%%%%%%%%%%
%%%%%%%%%%%%%%
%%%%%%%%%%%%%%
%%%%%%%%%%%%%%
%%%%%%%%%%%%%%

\subsection{$\updownarrow$-trees}\label{part2}
In a similar spirit to the previous section we can look at the law of our MBP, when issued from a single ancestor, conditioned to have a subtree of prolific individuals.  As such, for $A \in \mathcal{F}_t$, we define 
\begin{align}
 \P^\updownarrow_{\delta_{x}}(A | c_\emptyset(0) =\,\uparrow)&= \frac{\P^\updownarrow_{\delta_{x}}(A ; c_i =\,\uparrow, \text{ for at least one }i = 1, \dots, N_t)}{\P^\updownarrow_{\delta_{x}}(c_\emptyset(0) =\,\uparrow)}\notag\\
&= \frac{\E_{\delta_{x}}\left[\mathbf{1}_A\left(1 - \prod_{i=1}^{N_t}w(x_i(t))\right)\right]}{p(x)}.\label{probluetree}
\end{align}
In the next proposition, we want to describe our MBP under $\P^\updownarrow_{\delta_{x}}(\cdot | c_\emptyset(0) =\,\uparrow)$. In order to do so, we first need to introduce a type-$\uparrow$-type-$\downarrow$ MBP.

%\smallskip
%
%
%Next, we define the type-$\uparrow$-type-$\downarrow$ MBP process. As we will shortly see, it will turn out to be equivalent to  the previously defined $(\bP^\uparrow, G^\uparrow)$-MBP on to which we have grafted additional $(\bP^\downarrow, G^\downarrow)$-MBPs. 
%
\smallskip

Our type-$\uparrow$-type-$\downarrow$ MBP process, say $X^\updownarrow =(X^\updownarrow_t, t\geq0)$, has an ancestor which is of type-$\uparrow$. We will implicitly assume (and suppress from the notation $X^\updownarrow$) that $X^\updownarrow_0=\delta_x$, for $x \in E$. Particles in $X^\updownarrow$ of  type-$\uparrow$ move as a $\bP^\uparrow$-Markov process. When a branching event occurs for a type-$\uparrow$ particle, both type-$\uparrow$ and type-$\downarrow$ particles may be produced, but always at least one type-$\uparrow$ is produced. Type-$\uparrow$ particles may be thought of as offspring and any additional type-$\downarrow$ particles may be thought of as immigrants.  Type-$\downarrow$ particles that are created can only subsequently produce type-$\downarrow$ particles in such a way that they give rise to a $(\bP^\downarrow,G^\downarrow)$-MBP. 

\smallskip

The joint branching/immigration rate of type-$\uparrow$ and type-$\downarrow$ particles in  $X^\updownarrow$ at $x\in E$ is given by%\footnote{may not need second equaltiy}
\begin{align}
\varsigma^\updownarrow(x) 
&=\frac{\varsigma(x)}{p(x)}\,
\mathcal{E}_{x}\left[1-\prod_{j =1}^N w(x_j)\right].
%=\frac{\varsigma(x)}{p(x)}
%\mathcal{E}_{x}\Bigg[\sum_{\stackrel{I \subseteq \{1, \dots, N\}}{|I| \ge 1}}
%\prod_{i \in I}p(x_i)\prod_{i \in \{1, \dots, N\}\backslash I}w(x_i)
%\Bigg]
\label{bluebrate}
\end{align}
We can think of the branching rate in \eqref{bluebrate} as the original rate $\varsigma(x)$ multiplied by the probability (under $\mathcal{P}_x$) that at least one of the offspring are of type-$\uparrow$, given the branching particle is of type-$\uparrow$
\smallskip

At a branching/immigration event of a type-$\uparrow$ particle, we will write $N^\uparrow$ and $(x^\uparrow_i, i = 1, \cdots, N^\uparrow)$
for the number and positions of type-$\uparrow$ offspring and  
 $N^\downarrow$ and $(x^\downarrow_j, j = 1,\cdots, N^\downarrow)$  for the number and  positions of  type-$\downarrow$ immigrants. We will write $(\mathcal{P}^\updownarrow_{x}, x\in E)$ for the joint law of the the  random variables in the previous sentence.  Formally speaking, the branching generator, $G^\updownarrow$,  of offspring/immigrants for a type-$\uparrow$ particle positioned at $x\in E$ is written
\begin{equation}
G^\updownarrow[f,g](x) =\varsigma^\updownarrow(x)
\left(\mathcal{E}^{\updownarrow}_{x}\left[\prod_{i = 1}^{N^\uparrow} f(x^\uparrow_i)\prod_{j = 1}^{N^\downarrow} g(x^\downarrow_j)\right] - f(x)\right)
\label{Gupdown0}
\end{equation}
for $f,g\in L_\infty^{+,1}(E)$.

\smallskip

For our process $X^\updownarrow$, for each $x\in E$, we will define the laws  $\mathcal{P}^\updownarrow_{x}$ in terms of  an additional random selection from $(x_i, i= 1,\cdots, N)$ under  $\mathcal{P}_{x}$. Write 
$\mathcal{N}^\uparrow$ for the set of indices in $\{1,\cdots, N\}$, which, together, identify the type-$\uparrow$ particles, i.e. $(x_i, i\in \mathcal{N}^\uparrow) = (x^\uparrow_j, j = 1,\cdots, N^\uparrow)$. The remaining indices $\{1,\cdots, N\}\setminus\mathcal{N}^\uparrow$ will identify the type-$\downarrow$ immigrants from $(x_i, i= 1,\cdots, N)$. Thus, to describe $\mathcal{P}^\updownarrow_{x}$, for any $x\in E$, it suffices to  give the law of $(N; x_1,\dots,x_N ; \mathcal{N}^\uparrow)$. To this end, for $F\in \sigma(N; x_1,\dots,x_N)$ and $I\subseteq \N$, we will set
\begin{align}
%\int_{F\cap \{\mathcal N^\uparrow=I\}} \d
 \mathcal{P}^\updownarrow_{x}(F\cap \{\mathcal N^\uparrow=I\})
%\mathcal{P}^\updownarrow_{x}\Big( \{(x^\uparrow_i, i = 1, \cdots, N^\uparrow) = (x_i, i\in I)\}\cap F\Big)\notag\\
&:= 
\mathbf{1}_{\{  |I|\geq 1\}}
\frac{\varsigma(x)}{\varsigma^\updownarrow(x)p(x)}\,
\mathcal{E}_x\left[
\mathbf{1}_{F\cap\{I\subseteq \{1,\dots,N\}\}}\,
\prod_{i\in I}p(x_i)\prod_{i \in \{1, \dots, N\}\backslash I}w(x_i)\,\,
\right].
\label{PcalupdownCOM}
\end{align}
Said another way, 
%for any $x\in E$, $\mathcal{P}^\updownarrow_{x}$ gives the law of $(N; x_1,\dots,x_N ; \mathcal{N}^\uparrow)$, where $\mathcal{P}^\updownarrow_{x}(F):=\mathcal{P}_{x}(F)$ for all $F\in\sigma(N; x_1,\dots,x_N),$ and,
 for all $I\subseteq\N$, 
\begin{equation}
\mathcal{P}^\updownarrow_{x}(\mathcal N^\uparrow=I |  \sigma(N; x_1,\dots,x_N)) 
:= 
\mathbf{1}_{\{  |I|\geq 1\}\cap \{I\subseteq \{1,\dots,N\}\}}
\frac{\prod_{i\in I}p(x_i)\prod_{i \in \{1, \dots, N\}\backslash I}w(x_i)}
{1-\mathcal{E}_{x}\left[\prod_{j =1}^N w(x_j)\right]}.
\end{equation}

\smallskip

The pairs $(x^\uparrow_i, i = 1, \cdots, N^\uparrow)$ and $(x^\downarrow_j, j = 1,\cdots, N^\downarrow)$ under $(\mathcal{P}^\updownarrow_x, x\in E)$ in \eqref{Gupdown0} can thus be seen as equal in law to selecting the type of each particle following an independent sample of the non-local branching configuration 
%$(x_k,k = 1, \cdots, N)$ 
$(x_1,\dots, x_N)$ under $\mathcal{P}_{x}$,
where each $x_k$ is independently assigned either as type-$\uparrow$ with probability $p(x_k)$ or as type-$\downarrow$ with probability $w(x_k)=1-p(x_k)$, 
but then conditional on there being at least one type-$\uparrow$. 

\smallskip

As such with the definitions above, it is now a straightforward exercise to identify the branching generator in \eqref{Gupdown0} in terms of $(x_i, i = 1,\cdots, N)$ under $(\mathcal{P}_x, x\in E)$ via the following identity, for $x\in E$,
\begin{multline}
G^\updownarrow[f,g](x) 
= 
\frac{\varsigma(x)}{p(x)}
%\Bigg(
\mathcal{E}_{x} \Bigg[ 
\sum_{\stackrel{I \subseteq \{1, \dots, N\}}{|I|\geq 1}}
\prod_{i\in I}p(x_i)f(x_i)\prod_{i \in \{1, \dots, N\}\backslash I}w(x_i) g(x_i)
\Bigg] 
-\varsigma^\updownarrow(x)f(x)
%\mathcal{E}_{x}\Bigg[\sum_{\stackrel{I \subseteq \{1, \dots, N\}}{|I| \ge 1}}
%\prod_{i \in I}p(x_i)\prod_{i \in \{1, \dots, N\}\backslash I}w(x_i)
%\Bigg]
%\Bigg).
\label{Gupdown}
\end{multline}

\begin{prop}[Dressed $\uparrow$-trees]\label{new} For $x\in E$, the process $X^{\updownarrow}$ 
is equal in   law to  $X$ under $\P^\updownarrow_{\delta_{x}}(\cdot | c_\emptyset(0) =\,\uparrow)$. Moreover, both are equal in law to  to a dressed $(\ebP^\uparrow, G^\uparrow)$-MBP, say $X^\uparrow$, where
the motion semigroup $\ebP^\uparrow$ corresponds to the  Markov process $\xi$ on $E\cup\{\dagger\}$ with probabilities $(\mathbf{P}^\uparrow_x, x\in E)$ given by  (recalling that $p$ is valued 0 on $\dagger$)
\begin{equation}
\left.\frac{\d \mathbf{P}^\uparrow_{x}}{\d \mathbf{P}_{x}}\right|_{\sigma(\xi_s, s\leq t)} = \frac{p(\xi_t)}{p(x)}\exp\left(-\int_0^t \frac{G[w](\xi_s)}{p(\xi_s)}\d s\right), \qquad t\geq 0,
\label{pmg}
\end{equation}
and the branching generator is given by 
\begin{equation}
G^\uparrow[f] = \frac{1}{p}\left(G[pf + w] - (1-f)G[w]\right), \qquad f\in L_\infty^{+,1}(E).
\label{bluebmech}
\end{equation}
The dressing consists of  additional particles, which are immigrated non-locally in space at the branch points of $X^\uparrow$, 
with each immigrated particle continuing to evolve as an independent copy of $(X^\downarrow, \mathbb{P^\downarrow})$ from their respective space-point of immigration, such that the joint branching/immigration generator of type-$\uparrow$ offspring and type-$\downarrow$ immigrants is given by \eqref{Gupdown}.
\end{prop}

\begin{proof}[Proof of Proposition \ref{new}] We may think of  $((x_i(t), c_i(t)), i\leq N_t)$, $t\geq0$, under $\mathbb{P}^\updownarrow$ as a two-type branching process. 
To this end, let us write $N^\uparrow_t = \textstyle{\sum_{i =1}^{N_t} \mathbf{1}_{(c_i(t) = \,\uparrow)}}$ and $N^\downarrow = N_t -N_t^\uparrow$, for $t\geq 0$. Define, for $f,g\in L_\infty^{+,1}(E)$,
\begin{align}
u^\updownarrow_t[f, g](x)
&= \mathbb{E}^\updownarrow_{\delta_{x}}\left[\left.\Pi_t[f,g] \right| c_\emptyset(0) = \,\uparrow \right], \qquad t\geq 0.
\label{uupdown}
\end{align}
where, for $t\geq 0$,
\[
\Pi_t[f,g] =\prod_{i =1}^{N_t^\uparrow}p(x^\uparrow_i(t)) f(x^\uparrow_i(t))\prod_{j=1}^{N^\downarrow_t}w(x^\downarrow_j(t)) g(x^\downarrow_j(t)),
\]
where 
\[
(x^\uparrow_i(t), i = 1,\cdots, N^\uparrow_t) = (x_i(t) \text{ such that } c_i(t) = \uparrow, \,i\leq N_t)
\] and $(x^\downarrow_i(t), i = 1,\cdots, N^\downarrow)$ is similarly defined.

\smallskip

We can break the expectation in the definition of $u^\updownarrow_t[f, g]$ over the first branching event, noting that until that moment, the initial ancestor is necessarily prolific.  We have (again remembering  $p(\dagger)=0$)
\begin{align}
&u^\updownarrow_t[f, g](x)\notag \\
&= \frac{\mathbb{E}_{\delta_{x}}\left[\Pi_t[f,g] \mathbf{1}_{( c_\emptyset(0) = \,\uparrow )}\right]}{\mathbb{P}_{\delta_{x}}(c_\emptyset(0) = \,\uparrow )}\notag\\
&= \frac{1}{p(x)} \mathbf{E}_x\left[p(\xi_t)f(\xi_t){\rm e}^{-\int_0^t \varsigma(\xi_u)\d u}\right]\notag\\
%&+\int_0^{t\wedge\kappa^D_{r,\upsilon} } \frac{p(r+\upsilon s, \upsilon)}{p(x)} \sigma_{\texttt s}(r+\upsilon s,\upsilon){\rm e}^{-\int_0^s \sigma(r+\upsilon u, \upsilon)\d u}\notag\\
%&\hspace{2cm}
%\int_V u^\updownarrow_{t-s}[f, g](r+\upsilon s,\upsilon') \frac{p(r+\upsilon s,\upsilon')}{p(r+\upsilon s,\upsilon)}\pi_{\texttt s}(r+\upsilon s,\upsilon,\upsilon')\d \upsilon'\notag\\
&+\frac{1}{p(x)} \mathbf{E}_x\Bigg[ \int_0^{t} 
p(\xi_s)\frac{\varsigma(\xi_s)}{p(\xi_s)}{\rm e}^{-\int_0^s \varsigma(\xi_u)\d u}\notag\\
&
\hspace{3cm}\mathcal{E}_{\xi_s}
\Bigg[ 
\sum_{\stackrel{I \subseteq \{1, \dots, N\}}{|I|\geq 1}}
\prod_{i\in I}p(x_i)u^\updownarrow_{t-s}[f,g](x_i) \prod_{i \in \{1, \dots, N\}\backslash I}w(x_i) u^\downarrow_{t-s}[g](x_i)
\Bigg]\d s\Bigg].
\label{massive}
\end{align}
To help the reader interpret  \eqref{massive} better, we note that the first term on the right-hand side comes from the event that no branching occurs up to time $t$, in which case the initial ancestor is positioned at $\xi_t$. Moreover, we have used the fact that   $\mathbb{P}_{\delta_x}(c_\emptyset(0) = \uparrow|\mathcal{F}_t) = p(\xi_t)$. The second term is the consequence of a branching event occurring at time $s\in[0,t]$, at which point in time, the initial ancestor is positioned at $\xi_s$ and thus has  offspring  scattered at $(x_i, i =1 \cdots, N)$ according to $\mathcal{P}_{\xi_s}$. The contribution thereof from time $s$ to $t$, 
can be either captured by $u^\updownarrow_{t-s}[f,g]$, with probability $p$, if a given offspring is of type-$\uparrow$ (thereby growing a tree of particles marked both $\uparrow$ and $\downarrow$), or captured by $u^\downarrow_{t-s}[g]$, with probability $w$, if a given offspring is of type-$\downarrow$ (thereby growing a tree of particles marked only with $\downarrow$). Hence projecting the expectation of $\Pi_t[f,g]\mathbf{1}_{(c_\emptyset = \uparrow)}$ onto the given  configuration $(x_i, i =1 \cdots, N)$ at time $s$, we get the sum inside the expectation with respect to $\mathcal{P}_{\xi_s}$, which caters for all the possible markings of the offspring of the initial ancestor, ensuring that at least one of them is $\uparrow$ (which guarantees $c_\emptyset(0) = \uparrow$).
In both expectations, the event of killing is accommodated for the fact that $p(\dagger) =f(\dagger) = \varsigma(\dagger)= 0$.

%\smallskip
%
%Recalling \eqref{bluebrate} and \eqref{Gupdown}, see that, for $f,g\in L_\infty^{+,1}(E)$, %we define $H^\updownarrow[f,g]$ by
%\begin{align}
%&G^\updownarrow[f,g](x)+\varsigma^\updownarrow(x)f(x)\notag \\
%&:=
%\frac{\varsigma(x)}{p(x)}
% \mathcal{E}_{x}
%\Bigg[ 
%\sum_{\stackrel{I \subseteq \{1, \dots, N\}}{|I|\geq 1}}
%\prod_{i\in I}p(x_i)f(x_i)\prod_{i \in \{1, \dots, N\}\backslash I}w(x_i) g(x_i)
%\Bigg],
%\label{Gup}
%\end{align}
%for $x\in E$. %$r\in D$, $\upsilon\in D$.
\smallskip

We may now substitute %\eqref{betaup} 
and \eqref{Gupdown}   into \eqref{massive} to get
\begin{align}
&u^\updownarrow_t[f, g](x)\notag \\
&= \frac{1}{p(x)} \mathbf{E}_x\left[p(\xi_t)f(\xi_t){\rm e}^{-\int_0^t \varsigma(\xi_u)\d u}\right]\notag\\
&+\frac{1}{p(x)} \mathbf{E}_x\Bigg[ \int_0^{t} 
p(\xi_s)\frac{\varsigma(\xi_s)}{p(\xi_s)}{\rm e}^{-\int_0^s \varsigma(\xi_u)\d u}
%\notag\\
%&
%\hspace{3cm}
\left[G^\updownarrow[u^\updownarrow_{t-s}[f,g],u^\downarrow_{t-s}[g]](\xi_s)
+\varsigma^\updownarrow(\xi_s)u^\updownarrow_{t-s}[f,g](\xi_s)
\right]\d s\Bigg].
\label{massive2}
\end{align}

Next, recalling the first equality in \eqref{bluebrate} that $\varsigma(x) = \varsigma^\updownarrow(x) +  {G[w](x)}/{p(x)}$, in each of the terms on the right-hand side of \eqref{massive2}, we can exchange the exponential potential $\textstyle{\exp(-\int_0^\cdot \varsigma(\xi_u)\d u )}$ for the exponential potential $\textstyle{\exp(-\int_0^\cdot {G[w](\xi_u)}/{p(\xi_u)}\d u )}$ by transferring the the difference  in the exponent to an additive potential (cf. Lemma 1.2, Chapter 4 in \cite{Dynkin2}). In this exchange, the term  $\varsigma^\updownarrow(\xi_\cdot)u^\updownarrow_{t-s}[f,g](\xi_\cdot)$ is cancelled out on the right-hand side of \eqref{massive2}. Then recalling the change of measure  \eqref{pmg} that defines the semigroup $\bP^\uparrow$, we get,  on $E$,
\begin{align}
u^\updownarrow_t[f, g] &= \bP^\uparrow_t[f] + \int_0^t \bP^\uparrow_s\bigg[ G^\updownarrow[u^{\updownarrow}_{t-s}[f,g], u^\downarrow_{t-s}[g]] 
\bigg]\d s, \qquad t\geq0.
\label{updownevolve}
\end{align} 
%where, for $f\in L^+_\infty (E)$,
%\[
%\bS_p f (x)= \int_V [ f(r,\upsilon') - f(x)]\frac{p(r,\upsilon')}{p(x)}\pi_{\texttt{s}}(r,\upsilon,\upsilon')\d\upsilon'.
%\]
%Now, appealing to \eqref{exp2}, we can again transfer the remaining exponential potential (represented in the terms $p^{-1}\U_s[p]$, $s\leq t$) to an additive potential, giving us 
%\begin{align*}
%u^\updownarrow_t[f, g] &= \U_t[f] + \int_0^t \U_s\bigg[\bL^\uparrow u^\updownarrow_{t-s}[f,g]
%+ \left(\frac{\bS p}{p} + \frac{\bS w}{p} \right) u^\updownarrow_{t-s}[f,g]
%  \bigg]\d s\\
%&\hspace{1cm}+ \int_0^t \U_s\bigg[ H^\updownarrow[u^{\updownarrow}_{t-s}[f,g], u^\downarrow_{t-s}[g]] 
%+\left(\varsigma^\updownarrow -\varsigma + \frac{G[w]}{p}\right)u^{\updownarrow}_{t-s}[f,g] \bigg]\d s, \qquad t\geq0.
%\end{align*}
%Now that, as $\bS$ is a linear operator,  
%\[
%\bS p + \bS w= \bS (p+w) = \bS 1 = 0.
%\]
%Moreover, from \eqref{betaup}, we also have that $\varsigma^\updownarrow -\varsigma + p^{-1}G[w]=0$. It thus follows that 
%\begin{align}
%u^\updownarrow_t[f, g] &= \U_t[f] + \int_0^t \U_s\bigg[\bS_p u^\updownarrow_{t-s}[f,g]
%  \bigg]\d s+ \int_0^t \U_s\bigg[ H^\updownarrow[u^{\updownarrow}_{t-s}[f,g], u^\downarrow_{t-s}[g]]  \bigg]\d s, \qquad t\geq0.
%  \label{presplit}
%\end{align}
{\color{black}(Note, there is no need to define the object $\hat\bP^\uparrow$ in the sprit of \eqref{hatP} as the semigroup $\bP^\uparrow$ is that of a conservative process on $E$.)}
This is the semigroup of a two-type MBP in which $\downarrow$-marked particles immigrate off an $\uparrow$-marked MBP. We have yet to verify however that the $\uparrow$-marked MBP is in fact the previously described $(\bP^\uparrow, G^\uparrow)$-MBP.  In order to do this, we need to show that $G^\uparrow[f] =G^\updownarrow[f,1]$, for all $f\in L_\infty^{+,1}(E)$, where $G^\uparrow$ was given in \eqref{bluebmech}.

\smallskip

To this end, let us note two computational facts.  First, for any $x \in E$,
\begin{align}
1 &= \mathcal{E}_{x}\left[\prod_{i = 1}^N( p(x_i) + w(x_i))\right]= \mathcal{E}_{x}\left[\sum_{I \subseteq \{1, \dots, N\}}\prod_{i \in I}p(x_i)\prod_{i \in \{1, \dots, N\}\backslash I}w(x_i)\right],\label{e1}
\end{align}
so that
\begin{align}
\varsigma^\updownarrow(x) &= \frac{\varsigma(x)}{p(x)} \mathcal{E}_{x}\left[\sum_{\stackrel{I \subseteq \{1, \dots, N\}}{|I|\geq 1}}\prod_{i \in I}p(x_i)\prod_{i \in \{1, \dots, N\}\backslash I}w(x_i)\right] .
\label{e2}
\end{align}

\smallskip

Second, recalling \eqref{similar111}, note 
\[
G[f](x) -G[g](x) =  \varsigma(x)\left({\mathcal E}_x\left[\prod_{j =1}^N f(x_j)\right] - {\mathcal E}_x\left[\prod_{j =1}^N g(x_j)\right] \right)
\] 
and $G[1](x) \equiv 0$.  We thus have that, for $x\in E$,
\begin{align*}
G^\uparrow[f] &=G^\updownarrow[f,1](x)\\
&=\frac{\varsigma(x)}{p(x)}
 \mathcal{E}_{x}
\Bigg[ 
\sum_{\stackrel{I \subseteq \{1, \dots, N\}}{|I|\geq 1}}
\prod_{i\in I}p(x_i)f(x_i)\prod_{i \in \{1, \dots, N\}\backslash I}w(x_i) 
\Bigg]\\
&\hspace{2cm}- f(x)  \frac{\varsigma(x)}{p(x)}\mathcal{E}_{x}\left[\sum_{\stackrel{I \subseteq \{1, \dots, N\}}{|I| \ge 1}}\prod_{i \in I}p(x_i)\prod_{i \in \{1, \dots, N\}\backslash I}w(x_i)\right]
\\
&=\frac{\varsigma(x)}{p(x)}
 \mathcal{E}_{x}
\Bigg[ 
\prod_{k=1}^N \left(p(x_i)f(x_i)+w(x_k )\right) - \prod_{k=1}^N w(x_k )
\Bigg]\\
&\hspace{2cm}- f(x) \frac{\varsigma(x)}{p(x)} \Bigg[ 
\prod_{k=1}^N \left(p(x_i)+w(x_k )\right) - \prod_{k=1}^N w(x_k )
\Bigg]
\\
&=\frac{1}{p(x)}
\left\{
G[pf+w](x)-G[w](x)
\right\}
- f(x) \frac{1}{p(x)} 
\left\{
G[1](x)-G[w](x)
\right\}
\\
%%
%% Old formula (looks wrong!):
%%
%&= \frac{1}{1-w}\left(G[f(1-w) + w] + \varsigma f(1-w) - G[w]\right),
& = \frac{1}{p(x)}\left\{G[pf + w](x) - (1-f)G[w](x)\right\},
\end{align*}
since $p(x)+w(x)=1$, and this is just as required.
\end{proof}

\begin{theorem}[Skeletal decomposition]\label{skeleton} We assume throughout that (M1) and (M2) are in force. Suppose that $\mu = \textstyle{\sum_{i = 1}^n\delta_{x_i} }$, for $n\in\mathbb{N}$ and $x_1,\cdots, x_n\in E$. Then $(X,\mathbb{P}_\mu)$ is equal in law to 
\begin{equation}
%\Lambda_t : =
 \sum_{i = 1}^n \left(\emph{\texttt{B}}_i X^{i, \updownarrow}_t + (1- \emph{\texttt{B}}_i)X^{i, \downarrow}_t \right), \qquad t\geq 0,
\label{bernoulli}
\end{equation}
where, for each  $i = 1,\dots, n$,  $\emph{\texttt B}_i$ is an  independent Bernoulli random variable with probability of success given by 
\begin{equation}
\label{pdef}p(x_i) :=1-w(x_i)
\end{equation} and the processes $X^{i, \downarrow}$ and $X^{i,\updownarrow}$ are independent copies of $(X,\mathbb{P}^\downarrow_{\delta_{x_i}})$ and $(X, \P^\updownarrow_{\delta_{x_i}}(\cdot | c_\emptyset(0) =\,\uparrow))$, respectively.
\end{theorem}

As alluded to previously, Theorem \ref{skeleton} pertains to a classical decomposition of branching trees in which the process \eqref{bernoulli} describes how the MBP divides into the genealogical lines of descent which are `prolific' (surviving with probability $p$), in the sense that they create eternal subtrees which never leave the domain, and  those which are `unsuccessful' (dying with  probability $w$), in the sense that they generate subtrees in which all genealogies die out.
%\smallskip

\begin{remark}\label{iii}\rm It is an easy consequence of Theorem \ref{skeleton} that, 
for $t\geq 0$, the law of $X^\uparrow_t$  conditional on $\mathcal{F}_t =\sigma(X_s, s\leq t)$, is equal to that of a Binomial point process with intensity $p(\cdot)X_t(\cdot)$. The latter, written  ${\rm BinPP}(pX_t)$, is an atomic random measure given by 
\[
{\rm BinPP}(pX_t) = \sum_{i = 1}^{N_t}{\texttt B}_i\delta_{x_i(t)},
\]
where (we recall) that $\textstyle{X_t = \sum_{i=1}^{N_t}\delta_{x_i(t)}}$, and  ${\texttt B}_i$ is a Bernoulli random variable with probability $p(x_i(t))$, $i = 1,\cdots,N_t$.
\end{remark}

{\color{black}
\begin{remark}\label{1or2}\rm
It is also worth noting that the skeleton process $X^\uparrow$, given above, necessarily  has at least one type-$\uparrow$ offspring at each branch point, and indeed might have exactly one type-$\uparrow$ offspring (although possibly with other simultaneous type-$\downarrow$ immigrants). As such, an alternative way of looking at the  type-$\uparrow$ process  would be to  think of the skeleton of prolific individuals as a $(\bP^\Uparrow, G^\Uparrow)$-MBP with \emph{at least two} type-$\uparrow$ offspring at each branch point, with a modified motion $\bP^\Uparrow$ in place of $\bP^\uparrow$ which integrates the event of a single type-$\uparrow$ as an additional discontinuity in the movement. However, note these additional jumps are {special} in the sense as they are also potential points of simultaneous immigration of type-$\downarrow$ particles, unlike other jumps corresponding to $\bP^\uparrow$ where there is no type-$\downarrow$ immigration. 
\end{remark}
}

\subsection{Combining $\updownarrow$-trees and $\downarrow$-trees into the skeletal decomposition}
Finally we are now ready to give the skeletal decomposition of $(X, \mathbb{P})$.

\begin{proof}[Proof of Theorem \ref{skeleton}] As previously, we may think of  $((x_i(t), c_i(t)), i\leq N_t)$, $t\geq0$, as a two-type branching process under $\mathbb{P}^\updownarrow$. 
A similar calculation to  \eqref{updownCOM} gives us that, for  $\nu = \textstyle{\sum_{i = 1}^n \delta_{x_i}}$ with $n\geq 1$ and $x_i\in E$, $i = 1,\dots, n$, 
 \begin{align*}
&\E^\updownarrow_{\nu}\left[\Pi_t[f,g] \right]= \sum_{I\subseteq \{1,\dots, n\}}\prod_{i\in  I} {p(x_i)}{u^\updownarrow_t[f, g](x_i)}
\prod_{i\in \{1,\dots, n\}\backslash I}{w(x_i)}{u^\downarrow_t[wg](x_i)}.
\end{align*}

What this shows, together with the conditional version \eqref{updownCOM}, is that the change of measure \eqref{updownCOM} (which is of course unity) is equivalent to a Doob $h$-transform on a two-type branching particle system (i.e. types $\{\uparrow,\downarrow\}$) where we consider the system after disregarding the marks. The effect of this  Doob $h$-transform on type-$\downarrow$ particles is that they generate $(\bP^\downarrow, G^\downarrow)$-MBPs, where as  type-$\uparrow$ particles generate a dressed $(\bP^\uparrow,G^\uparrow)$-MBP as described in Proposition \ref{new}. 
\end{proof}

\subsection{Remarks on the skeletal decomposition for the NBP} {\color{black}The case of the skeletal decomposition for the NBP adds an additional layer of intricacy to the general picture given above. In this case, we have $E = D\times V$ with cemetery state $\dagger$ that is entered when there is neutron capture (a neutron disappears in $D\times V$ without undergoing fission) or neutrons go to the physical boundary points $\{(r,\upsilon): r\in \partial D \text{ and }\upsilon
\cdot{\bf n}_r>0\}$. It turns out that for the NBP, it is more convenient to view Theorem \ref{skeleton} in the spirit of Remark \ref{1or2}, i.e. we view the process $X^\uparrow$ as a branching process that has at least two offspring at every branching event and whose movement corresponds to advection plus an extra discontinuity, which accounts for a branching event with one offspring. 
\smallskip

To make this statement more precise, we first enforce the conditions of Theorem \ref{Kesten} in order to ensure (M1) and (M2) are satsfied. 
Indeed, on account of the inclusion $\{\zeta<\infty\}\subseteq \{W_\infty = 0\}$, we see that $w(x)\leq \P_{\delta_{x}}(W_\infty = 0)$, $r\in D, \upsilon\in V$. Recalling that $W$ converges both almost surely as well as in $L^1(\mathbb{P})$ to its limit, we have that $\P_{\delta_{x}}(W_\infty = 0)<1$ for $r\in D, \upsilon\in V$. This, combined with the fact that every particle may leave the bounded domain $D$ directly without scattering or undergoing fission with positive probability, gives us that  
\begin{equation}
{\rm e}^{-\int_0^{\kappa^D_{r, \upsilon}} \sigma(r+\upsilon s, \upsilon)\d s}<w(r, \upsilon)<1\text{ for all }r\in D, \upsilon\in V.
\label{<1}
\end{equation}
Note that the lower bound is uniformly bounded away from 0 thanks to the maximal diameter of $D$, the minimal velocity $\upsilon_{\texttt{min}}$ (which, together uniformly upper bound $\kappa^D_{r, \upsilon}$) and the uniformly upper bounded rates of fission and scattering. 
The upper inequality becomes an equality for $r\in \partial D$ and $\upsilon
\cdot{\bf n}_r>0$.
\smallskip

Now, viewing the NBP $X$ as a process with movement $\Q$ and branching generator $G$, heuristically speaking, we can understand a little better the the motions of $X^\uparrow$ and $X^\downarrow$ through the action of their generators. By considering only the leading order terms in small time (the process $(X_t, t\geq 0)$ is but a Markov chain), the action of the generator can be see as the result of the limit 
\begin{equation}
\bL f =\lim_{t\downarrow0}\frac{1}{t}\left(\Q_t[f] -f\right),
\label{genlim}
\end{equation}
for suitably smooth  $f$ (e.g. continuously differentiable within $L_\infty^{+}(D\times V)$).
It is easy to show, and indeed known (cf. e.g. \cite{MultiNTE, DL6}), that the action of the generator corresponding to $\Q$ is given by
\begin{equation}
\bL f(r,\upsilon) = \upsilon\cdot\nabla f(r,\upsilon) +\int_{V}\left( f(r, \upsilon') - f(r,\upsilon) \right)\sigma_{\texttt{s}} (r,\upsilon)\pi_{\texttt{s}}(r,\upsilon, \upsilon')\d\upsilon',
\label{L}
\end{equation}

for $f\in L^+_\infty(D\times V)$ such that $\nabla f$ is well defined (here $\nabla$ is assumed to act on the spatial variable $r$). 
{\color{black}
We emphasise again that, in view of Remark \ref{1or2}, this corresponds to motion plus a branching event with one offspring (or scattering). 
}
\smallskip

\smallskip

The change of measure \eqref{COMdown} induces a generator action given by 
\begin{align}
\bL^\downarrow f(r, \upsilon)&= \frac{1}{w(r, \upsilon)}\bL (w f)(r, \upsilon) + f (r, \upsilon)\frac{G[w]}{w}(r, \upsilon)\notag\\
& = \upsilon\cdot\nabla f(r,\upsilon) +\int_{V}\left( f(r, \upsilon') - f(r,\upsilon) \right)\sigma_{\texttt{s}} (r,\upsilon)\frac{w(r,\upsilon')}{w(r,\upsilon)}\pi_{\texttt{s}}(r,\upsilon, \upsilon')\d\upsilon'\notag\\
&\hspace{2cm}+ f(r, \upsilon)\left(\frac{\bL w}{w} + \frac{G[w]}{w}\right)(r, \upsilon) \notag\\
&= \upsilon\cdot\nabla f(r,\upsilon) +\int_{V}\left( f(r, \upsilon') - f(r,\upsilon) \right)\sigma_{\texttt{s}} (r,\upsilon)\frac{w(r,\upsilon')}{w(r,\upsilon)}\pi_{\texttt{s}}(r,\upsilon, \upsilon')\d\upsilon',
\label{Ldown}
\end{align}
where the fact that the right-hand side of \eqref{COMdown} is a martingale will lead to $\bL w + G[w] = 0$.

In other words, our heuristic reasoning above shows that the motion on the $\downarrow$-marked tree is tantamount to a $w$-tilting of the scatting kernel. This tilting favours scattering in a direction where 
extinction becomes more likely, and as such, $\bL^\downarrow$ encourages $\downarrow$-marked trees to become extinct `quickly'.
\smallskip

%
%   COMMENT
%
%\marginpar{\scriptsize\color{red}  DOUBLE CHECK! Whilst I believe this formula - its NOT same as in (2.12) of submitted paper SNTE-IIv9.pdf. Presumably because the $L^\uparrow$ in submitted paper did include the extra jumps from single prolific offspring fissions (as at end of (ii) of old theorem)}
%
%
%
Almost identical  reasoning shows that the change of measure \eqref{pmg} has generator with action
\begin{align}
\bL^\uparrow f(r, \upsilon) &= \frac{1}{p(r, \upsilon)}\bL(p f)(r, \upsilon) - f(r, \upsilon)\frac{G[w]}{p}(r, \upsilon)\notag\\
%&=\upsilon\cdot\nabla f(r,\upsilon) +\int_{V}\left( f(r, \upsilon') - f(r,\upsilon) \right)\sigma_{\texttt{s}} (r,\upsilon)\frac{p(r,\upsilon')}{p(r,\upsilon)}\pi_{\texttt{s}}(r,\upsilon, \upsilon')\d\upsilon'\notag\\
%&\hspace{2cm}+ f(r, \upsilon)\left(\frac{\bL p}{p} - \frac{G[w]}{p}\right)(r, \upsilon)\notag\\
%&=\upsilon\cdot\nabla f(r,\upsilon) +\int_{V}\left( f(r, \upsilon') - f(r,\upsilon) \right)\sigma_{\texttt{s}} (r,\upsilon)\frac{p(r,\upsilon')}{p(r,\upsilon)}\pi_{\texttt{s}}(r,\upsilon, \upsilon')\d\upsilon'\notag\\
%&\hspace{2cm}- f(r, \upsilon)\left(\frac{\bL w}{p} + \frac{G[w]}{p}\right)(r, \upsilon)\notag\\
&=\upsilon\cdot\nabla f(r,\upsilon) +\int_{V}\left( f(r, \upsilon') - f(r,\upsilon) \right)\sigma_{\texttt{s}} (r,\upsilon)\frac{p(r,\upsilon')}{p(r,\upsilon)}\pi_{\texttt{s}}(r,\upsilon, \upsilon')\d\upsilon',
\label{Lup}
\end{align}
for suitably smooth $f$, where we have again used $\bL w + G[w] = 0$ and left the calculations that the second equality from the first as an exercise for the reader. One sees again a $p$-tilting of the scattering kernel, and hence $L^\uparrow$ rewards scattering in directions that `enable survival'. Note, moreover for regions of $D\times V$ for which $p(r,\upsilon)$ can be come arbitrarily small (corresponding to a small probability of survival), the scattering rate also becomes very large, and hence $L^\uparrow$ `urgently' scatters particles away from such regions.

\subsection{Remarks on BBM}
On account of the fact that we have stated Theorem \ref{skeleton} for a relatively general MBP with non-local branching, it is worth pointing to the known example of a BBM in a strip that has previously been worked out in detail in \cite{HHK}. %As with the setting of the NBP, little will be known about the spatial extinction probability $w$, other than its martingale properties, and so the skeletal decomposition carries more value as an tool for proving other results (as indeed we will use it  for the NBP), rather than a means to an end.
This model has the features that $\bP$ is that of a Brownian motion with drift $\mu$ killed on existing an interval $[0,K]$, so that $\bL = (1/2)\d^2/\d x^2 + \mu \d /\d x$, the branching rate $\varsigma$ is constant (not spatially dependent) and the offspring distribution is concentrated at the point of death of each particle. As such, the generator $G$ in \eqref{similar111} takes the simpler form 
\begin{equation}
G[\theta] = \varsigma\mathcal{E}\left(\theta^N - \theta\right) 
%= \varsigma\left(\sum_{k \geq 0}\theta^k p_k -\theta\right)
\label{BBMG}
\end{equation}
%$G[\theta] = \varsigma\mathcal{E}\left(\theta^N - \theta\right)$,
where %$(p_k,k\geq 0)$ is distribution of the number of offspring. 
it suffices to take $\theta$ as a number in $(0,1)$, rather than a function, as there is no spatial dependency. The extinction probability now solves the differential equation 
\[
\frac{1}{2}\frac{\d w}{\d x^2} +\mu\frac{\d w}{\d x}+ G[w] = 0 \text{ on }(0,K)\text{ with } w(0)= w(K) = 1.
\]
In order for survival to occur with positive probability, it is required that the leading eigenvalue of the mean semigroup associated to the branching process, which is  $\lambda_*  := (m-1)\varsigma - \mu^2/2 -\pi^2/2K$, must satisfy $\lambda_*>0$, where $\textstyle{m = \sum_{k= 0}^\infty kp_k}$ is the mean number of offspring. Note, the mean semigroup is the analogue of \eqref{semigroup} and the leading eigenvalue plays precisely the role of $\lambda_*$ in Theorem \ref{PF} for the NTE. 
\smallskip

For the $\downarrow$ process, writing $G^\downarrow$ in \eqref{redbmech} in a similar format to \eqref{BBMG}, it is straightforward to verify that it agrees with the branching generator stipulated in analysis of the red tree given in \cite{HHK}.
\smallskip

However, for the $\uparrow$ process, this model also takes the point of view described in Remark \ref{1or2}. Indeed, it is straightforward to show that the branching generator for the blue tree in \cite{HHK} agrees with $G^\Uparrow$ given in Remark \ref{1or2} and the `discontinuity' associated with a birth of one offspring is appended to the motion. However, since this model only has local branching and the movement is a Brownian motion, this does not actually change the motion. On the other hand, this choice does affect the overall process $X^{\updownarrow}$ and leads to two types of immigration of red trees onto the blue tree: immigration at branch points and immigration along the trajectory, with the latter immigration occurring at the points corresponding to a `birth of one offspring'.
\smallskip

%Writing the generators $G^\uparrow$ in \eqref{bluebmech} and $G^\downarrow$ in \eqref{redbmech} in a similar format to \eqref{BBMG}, it is straightforward to verify that they agree with those stipulated in skeletal decomposition given in Section 3 of \cite{HHK}. Similarly the movement semigroups $\bP^\uparrow$ and $\bP^\downarrow$ agree precisely.
%\smallskip

When, additionally, the interval $[0,K]$ is replaced by $\mathbb{R}$, the extinction probability $w$ is no longer spatially dependent and is a simple solution of $G[w] = 0$. Assuming that $w\in(0,1)$, it is easy to see that  $\bL^\uparrow$ and $\bL^\downarrow$ are both equal to $\bL$ and the skeletal decomposition is nothing more than the original skeletal decomposition for Galton--Watson processes (albeit in continuous time) given in the book of  Harris \cite{H}.}

%%%%%%%%%%%%%%%%%%%%%%%%%%
%%%%%%%%%%%%%%%%%%%%%%%%%%
%%%%%%%%%%%%%%%%%%%%%%%%%%
%%%%%%%%%%%%%%%%%%%%%%%%%%
%%%%%%%%%%%%%%%%%%%%%%%%%%
%%%%%%%%%%%%%%%%%%%%%%%%%%
%%%%%%%%%%%%%%%%%%%%%%%%%%
%%%%%%%%%%%%%%%%%%%%%%%%%%
%%%%%%%%%%%%%%%%%%%%%%%%%%
%%%%%%%%%%%%%%%%%%%%%%%%%%
%%%%%%%%%%%%%%%%%%%%%%%%%%
%%%%%%%%%%%%%%%%%%%%%%%%%%
%%%%%%%%%%%%%%%%%%%%%%%%%%
%%%%%%%%%%%%%%%%%%%%%%%%%%
%%%%%%%%%%%%%%%%%%%%%%%%%%
%%%%%%%%%%%%%%%%%%%%%%%%%%
%%%%%%%%%%%%%%%%%%%%%%%%%%
%%%%%%%%%%%%%%%%%%%%%%%%%%
%%%%%%%%%%%%%%%%%%%%%%%%%%
%%%%%%%%%%%%%%%%%%%%%%%%%%
%%%%%%%%%%%%%%%%%%%%%%%%%%
%%%%%%%%%%%%%%%%%%%%%%%%%%
%%%%%%%%%%%%%%%%%%%%%%%%%%
%%%%%%%%%%%%%%%%%%%%%%%%%%
%%%%%%%%%%%%%%%%%%%%%%%%%%
\section{SLLN on the skeleton}

Our aim is to use the skeletal decomposition of the neutron branching process to prove Theorem \ref{SLLN} by first stating and proving the analogous result for $X^\uparrow$. {\color{black}Hence, in what follows, we will assume (H1), (H2)$^*$, (H3)$^*$ and (H4) hold.} Before continuing to the proof, let us consider a useful identity. For a suitable $g\in L_\infty(D\times V)$  and $t \ge 0$, we have from Theorem \ref{skeleton} (cf. Remark \ref{iii}) that 
\begin{align}
\mathbb{E}_{\delta_{(r, \upsilon)}}^\uparrow\left[ \langle g, X_t^\uparrow\rangle \right] &= \mathbb{E}_{\delta_{(r, \upsilon)}}^\updownarrow\left[\langle g, pX_t\rangle \big|  c_\emptyset(0) = \uparrow \right] = \frac{1}{p(r, \upsilon)}\mathbb{E}_{\delta_{(r, \upsilon)}}\left[\langle gp, X_t \rangle\right] \label{identity1}.
\end{align}

We can use this identity to show that ${\lambda_*}$ is also an eigenvalue for the linear semigroup of $X^\uparrow$, as well as to compute the associated  left and right eigenfunctions (in a similar sense to \eqref{leftandright}). Our first claim is that the right eigenfunction is given by $\varphi/p$. Indeed, for $(r, \upsilon) \in D \times V$, due to the above computation,
\begin{equation}
\mathbb{E}_{\delta_{(r, \upsilon)}}^\uparrow[\langle {\varphi}/{p}, X_t^\uparrow \rangle] = \frac{\mathbb{E}_{\delta_{(r, \upsilon)}}[\langle \varphi, X_t\rangle]}{p(r, \upsilon)} = {\rm e}^{\lambda_* t}\frac{\varphi(r, \upsilon)}{p(r, \upsilon)}. 
\label{bboneright}
\end{equation}
For the left eigenfunction, again using~\eqref{identity1}, we have
\begin{equation}
\langle \tilde\varphi p, \mathbb{E}_{\delta_\cdot}^\uparrow[\langle g, X_t^\uparrow\rangle]\rangle = \langle  \tilde\varphi p, {\mathbb{E}_{\delta_\cdot}[\langle gp, X_t\rangle]}/{p(\cdot)} \rangle = \langle  \tilde\varphi, \mathbb{E}_{\delta_\cdot}[\langle gp, X_t\rangle]\rangle = {\rm e}^{\lambda_* t}\langle \tilde{\varphi}p, g \rangle. \label{bboneleft}
\end{equation}
Hence $\tilde\varphi p$ is the corresponding left eigenfunction with eigenvalue ${\rm e}^{\lambda_*t}$. 

\smallskip

It now follows by similar arguments to those given in~\cite{SNTE} that
\begin{equation}
W_t^\uparrow \coloneqq {\rm e}^{-\lambda_* t}\frac{\langle \varphi/p , X_t^\uparrow\rangle}{\langle \varphi/ p, \mu\rangle}, \quad t \ge 0,
\label{bbonemg}
\end{equation}
is a positive martingale under $\mathbb{P}_{\mu}^\uparrow$ for $\mu\in \mathcal{M}(D\times V)$, and hence has a finite limit, which we denote $W_\infty^\uparrow$.

\smallskip
{\color{black}

A second useful fact that we will use is the following result.
\begin{lem}\label{Simonlemma}
There exists a constant $C\in(0,\infty)$ such that  $\textstyle{\sup_{r\in D,\upsilon\in V}\varphi(r,\upsilon)/p(r,\upsilon)}<C$.
\end{lem}
\begin{proof}
Let us introduce the family of measures $\P^\varphi: = (\P^\varphi_\mu, \mu\in\mathcal{M}(D\times V))$, where
\begin{equation}
\left.\frac{\d \mathbb{P}^\varphi_{\mu}}{\d\mathbb{P}_{\mu}}\right|_{\mathcal{F}_t}  =
 W_t, \qquad t\geq 0,
\label{mgCOM}
\end{equation}We start by noting that, for all $r\in D, \upsilon\in V$,
$p(r,\upsilon) = 1-   \P_{\delta_{(r,\upsilon)}}(\zeta<\infty)=  \mathbb{P}_{\delta_{(r,\upsilon)}}(X \text{ survives})$, where $\zeta$ is the lifetime of $X$ defined in \eqref{zeta}. 
Taking account of \eqref{mgCOM}, we can thus write, with the help of Fatou's Lemma and Jensen's inequality,
\begin{align*}
p(r,\upsilon)& = \lim_{t\to\infty} \mathbb{P}_{\delta_{(r,\upsilon)}}(t<\zeta)\\
&=   \lim_{t\to\infty}\mathbb{E}^\varphi_{\delta_{(r,\upsilon)}}\left[\frac{1}{W_t}\mathbf{1}_{(t<\zeta)}\right]\\
&\geq \mathbb{E}^\varphi_{\delta_{(r,\upsilon)}}\left[{1}/{W_\infty}\right]\\
&\geq {1}/{\mathbb{E}^\varphi_{\delta_{(r,\upsilon)}}\left[W_\infty\right]},
\end{align*}
where we note that the indicator is dropped in the first inequality as, from Lemma 6.1 in \cite{SNTE}, the process $(X,\P^\varphi)$ is immortal.

\smallskip

From equations (10.1) and (10.3) in \cite{SNTE}, it has already been shown that 
\begin{equation}
\mathbb{E}^\varphi_{\delta_{(r,\upsilon)}}\left[W_\infty\right] = \lim_{t\to\infty} 
\mathbb{E}_{\delta_{(r,\upsilon)}}\left[W_t^2\right]\leq  c \int_0^\infty 
{\rm e}^{-2\lambda_* t } \frac{\psi_t[1](r, \upsilon)}{\varphi (r,\upsilon)} \d t,
\label{usebelow}
\end{equation}
for some constant $c\in(0,\infty)$. Taking account of  Theorem \ref{PF}, which tells us that 
\[
\lim_{t\to\infty} {\rm e}^{-\lambda_*t}\psi_t[1](r, \upsilon) = \langle1,\tilde\varphi\rangle\varphi(r,\upsilon)\leq \norm{\varphi}_\infty\langle1,\tilde\varphi\rangle,\]
 we deduce that there exists a constant $C\in(0,\infty)$, which does not depend on $(r,\upsilon)\in D\times V$, such that 
\[
p(r,\upsilon)\geq \frac{\varphi(r,\upsilon)}{C}.
\]
The result now follows.
\end{proof}
}
We are now in a position to state and prove a strong law for the skeleton $X^\uparrow$. 
\begin{theorem}\label{SLLNbbone} 
For all non-negative and directionally continuous $g$ (in the sense that $\textstyle{\lim_{s\downarrow0}g(r+\upsilon s, \upsilon) = g(r,\upsilon)}$ for all $r\in D$, $\upsilon\in V$) such that, for some constant $c>0$,  $g\leq  c\varphi / p$, 
\begin{equation}
\lim_{t\to\infty} {\rm e}^{-\lambda_* t}\langle g, X_t^\uparrow\rangle = \langle g,\tilde{\varphi}p\rangle\langle \varphi/ p, \mu\rangle W^\uparrow_\infty.
\label{onlaticefirst}
\end{equation}
$\mathbb{P}_{\mu}^\uparrow$-almost surely for $\mu\in \mathcal{M}(D\times V)$
\end{theorem}

We prove this theorem by breaking it up into several parts. Starting with the following lemma, we first prove that Theorem \ref{SLLNbbone} holds along lattice times. Our proofs are principally by techniques that have been used a number of times in the literature, developed by \cite{AH, EHK, CRY, CRSZ} amongst others. Just before we state the next lemma, it will be convenient  to quickly introduce the notation  $\mathcal{F}^\uparrow_t  =\sigma(X^\uparrow_s : s\leq t)$, $t\geq 0$.

\begin{lem}\label{lemma1}
Fix $\delta > 0$. For non-negative bounded functions $g\in L^+_\infty(D\times V)$, define
\begin{equation}
U_t = {\rm e}^{-\lambda_*t}\langle {g\varphi}/{p}, X_t^\uparrow \rangle, \quad t \ge 0.
\label{utg}
\end{equation}
Then, for any non-decreasing sequence $(m_n)_{n \ge 0}$ with $m_0>0$ and $(r, \upsilon) \in D\times V$,
\begin{equation}
\lim_{n \to \infty}|U_{(m_n + n)\delta} - \mathbb{E}^\uparrow[U_{(m_n + n)\delta} | \mathcal{F}_{n\delta}^\uparrow]| = 0, \quad \mathbb{P}^\uparrow_{\delta_{(r, \upsilon)}}\text{-a.s.}
\label{limit1}
\end{equation} 
%where $\mathcal{F}^\uparrow_t  =\sigma(X^\uparrow_s : s\leq t)$, $t\geq 0$.
\end{lem}
\begin{proof}

By the Borel-Cantelli lemma, it is sufficient to prove that for each $(r, \upsilon) \in D \times V$ and all $\varepsilon >0$,
\begin{equation}
\sum_{n\ge 1}\mathbb{P}^\uparrow_{\delta_{(r, \upsilon)}}\left(\big|U_{(m_n + n)\delta} - \mathbb{E}[U_{(m_n + n)\delta}|\mathcal{F}_{n\delta}^\uparrow]\big| > \varepsilon\right) <\infty.
\label{enoughforBC1}
\end{equation}

\smallskip

To this end, note that Markov's inequality gives
\begin{align}
&\mathbb{P}^\uparrow_{\delta_{(r, \upsilon)}}\left(\big|U_{(m_n + n)\delta} - \mathbb{E}[\right.U_{(m_n + n)\delta}|\mathcal{F}_{n\delta}^\uparrow]\big| > \varepsilon\left.\right)\notag
\\
&\hspace{2cm} 
\le \varepsilon^{-2}\mathbb{E}_{\delta_{(r, \upsilon)}}^\uparrow\left(\big|U_{(m_n + n)\delta} - \mathbb{E}[U_{(m_n + n)\delta}|\mathcal{F}_{n\delta}^\uparrow]\big|^2\right). 
\label{enoughforBC2}
\end{align}
Hence, let us consider the term in the conditional expectation on the right-hand side above. First note that 
\begin{equation}
U_{(m_n + n)\delta} - \mathbb{E}^\uparrow[U_{(m_n + n)\delta} | \mathcal{F}_{n\delta}^\uparrow] = \sum_{i=1}^{N_{n\delta}}{\rm e}^{-n\delta\lambda_*}( U_{m_n \delta}^{(i)} - \mathbb{E}^\uparrow[U_{m_n\delta}^{(i)} | \mathcal{F}_{n\delta}^\uparrow] ),
\label{branchprop}
\end{equation}
where, given $\mathcal{F}_t^\uparrow$, the $U^{(i)}$ are independent and equal in distribution to $U$ under $\mathbb{P}_{\delta_{(R_i(t), \Upsilon_i(t))}}^\uparrow$ and $\{(R_i(t), \Upsilon_i(t)): i = 1,\cdots, N_t\}$ describes the configuration of $X^\uparrow$ at time $t \ge 0$.
Note in particular, conditional on $\mathcal{F}_{n\delta}^\uparrow$, $Z_i = U_{m_n \delta}^{(i)} - \mathbb{E}^\uparrow(U_{m_n\delta}^{(i)} | \mathcal{F}_{n\delta}^\uparrow)$ are independent with $\mathbb{E}[Z_i] = 0$. The formula for the variance of sums of zero mean independent random variables %give us For $q \in [1,2]$, a commonly used inequality for sums of zero-mean iid random variables in~\cite{Big92} states that
%\begin{equation}
%\mathbb{E}^\uparrow\left[\bigg| \sum_{i = 1}^{n}Z_i\bigg|^q\right] \le 2^q \sum_{i = 1}^n\mathbb{E}|Z_i|^q.
%\label{ineq1}
%\end{equation}
%Combining this inequality for $q = 2$ 
together with the inequality $|a+b|^{2}\leq 2( |a|^{2} + |b|^{2})$, we get
\begin{align*}
\mathbb{E}^\uparrow&(|U_{(m_n + n)\delta} - \mathbb{E}[U_{(m_n + n)\delta}|\mathcal{F}_{n\delta}^\uparrow]|^2|\mathcal{F}_{n\delta}^\uparrow) \\
&= \sum_{i = 1}^{N_{n\delta}}{\rm e}^{-2\lambda_* n\delta}\mathbb{E}^\uparrow\left[\left.\bigg| U_{m_n \delta}^{(i)} - \mathbb{E}^\uparrow[U_{m_n\delta}^{(i)} | \mathcal{F}_{n\delta}^\uparrow] \bigg|^2\right|  \mathcal{F}_{n\delta}^\uparrow\right]\\
&\le  \sum_{i = 1}^{N_{n\delta}}{\rm e}^{-2\lambda_* n \delta} \mathbb{E}^\uparrow\left[4(| U_{m_n \delta}^{(i)}|^2 + |\mathbb{E}^\uparrow[U_{m_n\delta}^{(i)} | \mathcal{F}_{n\delta}^\uparrow]|^2) | \mathcal{F}_{n\delta}^\uparrow \right]\\
& \le 4\sum_{i = 1}^{N_{n\delta}} {\rm e}^{-2\lambda_* n \delta}\mathbb{E}^\uparrow\left[|U_{m_n\delta}^{(i)}|^2 | \mathcal{F}_{n\delta}\right],
\end{align*}
where we have used Jensen's inequality again in the final inequality.
Hence, with $\{(R_i(n\delta), \Upsilon_i(n\delta)): i = 1, \dots, N_{n\delta}\}$ describing the configurations of the particles at time $N_{n\delta}$ in $X^\uparrow$, we have
\begin{align}
\sum_{n = 1}^\infty & \mathbb{E}^\uparrow\left[ |U_{(m_n + n)\delta} - \mathbb{E}^\uparrow(U_{(m_n + n)\delta}|\mathcal{F}_{n\delta}^\uparrow)|^2 %| \mathcal{F}_{n\delta}^\uparrow
\right] \notag\\
& \le 4\sum_{n = 1}^\infty {\rm e}^{-2\lambda_*n\delta}\mathbb{E}^\uparrow_{\delta_{(r, \upsilon)}}\left[ \sum_{i = 1}^{N_{n\delta}}\mathbb{E}^\uparrow_{\delta_{(R_i(n\delta), \Upsilon_i(n\delta))}}\left[ U_{m_n\delta}^2\right]\right] \notag\\ 
& \le 4\norm{g}_\infty^2\sum_{n = 1}^\infty{\rm e}^{-2\lambda_*n\delta} \mathbb{E}^\uparrow_{\delta_{(r, \upsilon)}}\left[ \sum_{i = 1}^{N_{n\delta}}\frac{\varphi(R_i(n\delta), \Upsilon_i(n\delta))^2}{p(R_i(n\delta), \Upsilon_i(n\delta))^2}\mathbb{E}^\uparrow_{\delta_{(R_i(n\delta), \Upsilon_i(n\delta))}}\left[ ({W}_{m_n\delta}^{\uparrow})^2\right]\right], \label{BC1}
\end{align}
where the final inequality was obtained by noting that, from the definitions of $U_t$ and $W_t^\uparrow$, we have
\[
\mathbb{E}^\uparrow_{\delta_{(r, \upsilon)}}[U_t^2] \le \Vert g \Vert_\infty^2\frac{\varphi(r, \upsilon)^2}{p(r, \upsilon)^2}\mathbb{E}^\uparrow_{\delta_{(r, \upsilon)}}[(W_t^\uparrow)^2].
\]

\smallskip 

Due to Theorem \ref{skeleton}, in particular Remark \ref{iii}, and the calculation leading to \eqref{identity1}, we have, for all $t\geq 0$,
{\color{black}\begin{align}
\mathbb{E}_{\delta_{(r, \upsilon)}}^\uparrow[(W_t^\uparrow)^2] &=
\frac{{\rm e}^{-2\lambda_*t}}{(\varphi(r, \upsilon)/p(r, \upsilon))^2} \mathbb{E}_{\delta_{(r, \upsilon)}}^\uparrow [\langle \varphi/p, X_t^\uparrow\rangle^2 ] \notag \\
&=
\frac{{\rm e}^{-2\lambda_*t}}{(\varphi(r, \upsilon)/p(r, \upsilon))^2} \mathbb{E}_{\delta_{(r, \upsilon)}}^\updownarrow \left[\langle \varphi/p, {\rm BinPP}(pX_t)\rangle^2 | c_\emptyset(0) = \uparrow\right] \notag \\
&\leq \frac{p(r, \upsilon)^2}{\varphi(r, \upsilon)^2} \left\{
{\rm e}^{-2\lambda_*t}\mathbb{E}_{\delta_{(r, \upsilon)}}\left[\langle \varphi^2/p, X_t\rangle \right]/p(r, \upsilon)
+
{\rm e}^{-2\lambda_*t}\mathbb{E}_{\delta_{(r, \upsilon)}}\left[\langle \varphi, X_t\rangle^2 \right]/p(r, \upsilon)
\right\}
\notag \\
&\leq
 C\left( \frac{{\rm e}^{-\lambda_*t}}{\varphi(r,\upsilon)}
+\mathbb{E}_{\delta_{(r, \upsilon)}}\left[W^2_t \right]
 \right)
p(r, \upsilon)\notag
\\
&\leq C\frac{p(r,\upsilon)}{\varphi(r,\upsilon)}\left(\varphi(r,\upsilon)\mathbb{E}_{\delta_{(r, \upsilon)}}\left[W^2_t \right]
+
1
\right)
 \label{mgchange}
\end{align}
where we have used Lemma \ref{Simonlemma} in the second inequality. From Corollary 5.3 of \cite{SNTE}, more precisely from its proof, we know that 
$\textstyle{\mathbb{E}_{\delta_{(r,\upsilon)}}[\sup_{t\geq 0}W_t^2}]<\infty$. Hence we have from Doob's maximal inequality that, for each fixed $t\geq 0$,
\begin{align}
\mathbb{E}_{\delta_{(r, \upsilon)}}^\uparrow\left[(W_{t}^\uparrow)^2\right]&\leq \mathbb{E}_{\delta_{(r, \upsilon)}}^\uparrow\left[\sup_{s\geq 0}(W_{s}^\uparrow)^2\right]\notag\\
&\leq \limsup_{s\to\infty}4\mathbb{E}_{\delta_{(r, \upsilon)}}^\uparrow\left[(W_{s}^\uparrow)^2\right]\notag\\
&\leq4C\frac{p(r,\upsilon)}{\varphi(r,\upsilon)}\left(\varphi(r,\upsilon)\mathbb{E}_{\delta_{(r, \upsilon)}}\left[W^2_\infty \right]
+
1
\right)\notag\\
&\leq4C\frac{p(r,\upsilon)}{\varphi(r,\upsilon)}\left(C'
+
1
\right)<\infty
\label{arrowL2}
\end{align} 
for some constant $C'$ which does not depend on $(r,\upsilon)$, where we have used \eqref{usebelow}.
(Note \eqref{arrowL2} implies that $W^\uparrow$ is an $L_2(\P^\uparrow)$-convergent martingale.)
\smallskip

Substituting the estimate  \eqref{mgchange} back into  \eqref{BC1} and making use of the uniform boundedness of $\varphi$, we get
\begin{align}
\sum_{n = 1}^\infty & \mathbb{E}^\uparrow\left[ |U_{(m_n + n)\delta} - \mathbb{E}^\uparrow(U_{(m_n + n)\delta}|\mathcal{F}_{n\delta}^\uparrow)|^2 \right] \notag\\
&\le K\norm{g}_\infty^2\sum_{n = 1}^{\infty}{\rm e}^{-2\lambda_* n\delta}\mathbb{E}_{\delta_{(r, \upsilon)}}^\uparrow\Bigg[\sum_{i = 1}^{N_{n\delta}}\frac{\varphi(R_i(n\delta), \Upsilon_i(n\delta))}{p(R_i(n\delta), \Upsilon_i(n\delta))}\Bigg], \label{bound1}
\end{align}
for some constant $K\in(0,\infty)$.
Now the fact that $\varphi/p$ is an eigenfunction for the linear semigroup of $X^\uparrow$, we get
\begin{align}
\sum_{n = 1}^\infty  \mathbb{E}^\uparrow\left[ |U_{(m_n + n)\delta} - \mathbb{E}^\uparrow(U_{(m_n + n)\delta}|\mathcal{F}_{n\delta}^\uparrow)|^2 \right] 
&\le K\norm{g}_\infty^2\sum_{n \ge 1}{\rm e}^{-2\lambda_* n\delta}\mathbb{E}_{\delta_{(r, \upsilon)}}^\uparrow[\langle \varphi/p, X_{n\delta}^\uparrow\rangle] \notag \\
&= K\norm{g}_\infty^2\frac{\varphi(r,\upsilon)}{p(r, \upsilon)}\sum_{n \ge 1}{\rm e}^{-\lambda_* n\delta}<\infty. \label{bound2}
\end{align}
The result now follows by \eqref{enoughforBC1} and \eqref{enoughforBC2}. 
}
\end{proof}

{\color{black}
It is worth noting that a small corollary falls out of the above proof, which will be useful later on.
\begin{corollary}\label{freebie}
We have $\textstyle{\sup_{t\geq 0}W^\uparrow_t}$ is square integrable and hence $W^\uparrow$ converges in $L_2(\mathbb{P}^\uparrow)$.
\end{corollary}
}

\begin{proof}[Proof of Theorem~\ref{SLLNbbone} (lattice sequences)] 
We have already noted that 
\[
\mathbb{E}_{\delta_{(r, \upsilon)}}^\uparrow\left[U_{t+s}\big| \mathcal{F}_t^\uparrow \right] = \sum_{i = 1}^{N_t}{\rm e}^{-\lambda_* t}\bar U_s^{(i)},
\]
where, given $\mathcal{F}_t^\uparrow$, the $\bar U_s^{(i)}$ are independent and equal  to  $\mathbb{E}_{\delta_{(R_i(t), \Upsilon_i(t))}}^\uparrow[U_s]$ and $\{(R_i(t), \Upsilon_i(t)): i = 1,\cdots, N_t\}$ describes the configuration of $X^\uparrow$ at time $t \ge 0$. Hence, once again using \eqref{identity1}, as well as \eqref{semigroup}, we have
\begin{align}
\mathbb{E}^\uparrow_{\delta_{(r, \upsilon)}}\left[U_{t+s} | \mathcal{F}_t^\uparrow \right]
 &= \sum_{i = 1}^{N_t}{\rm e}^{-\lambda_* t}\mathbb{E}_{\delta_{(R_i(t), \Upsilon_i(t))}}^\uparrow \left[{\rm e}^{-\lambda_* s}\langle g\varphi/p, X_s^\uparrow\rangle \right] \notag \\
&= \sum_{i = 1}^{N_t}{\rm e}^{-\lambda_* t}\frac{\mathbb{E}_{\delta_{(R_i(t), \Upsilon_i(t))}}[{\rm e}^{-\lambda_* s}\langle g\varphi, X_s\rangle]}{p(R_i(t), \Upsilon_i(t))}\notag \\
&= \sum_{i = 1}^{N_t}{\rm e}^{-\lambda_* t}\frac{{\rm e}^{-\lambda_* s}\psi_s[\varphi g](R_i(t), \Upsilon_i(t))}{p(R_i(t), \Upsilon_i(t))}
 \notag \\
&= \frac{p(r, \upsilon)}{\varphi(r, \upsilon)} \langle g\varphi, \tilde\varphi\rangle W_t^\uparrow \notag\\
&\hspace{1cm}+ \sum_{i = 1}^{N_t}{\rm e}^{-\lambda_* t}
\left( {\rm e}^{-\lambda_* s}\frac{\psi_s[\varphi g](R_i(t), \Upsilon_i(t))}{\varphi(R_i(t), \Upsilon_i(t))} - \langle g\varphi, \tilde\varphi\rangle \right)
\frac{\varphi(R_i(t), \Upsilon_i(t))}{p(R_i(t), \Upsilon_i(t))}. \label{step1}
\end{align}
Appealing to Theorem \ref{PF}, we can pick  $s $ sufficiently large so that, for any given $\varepsilon>0$,
\begin{equation}
\norm{  {\rm e}^{-\lambda_* s}\varphi^{-1}\psi_s[\varphi g] -\langle\tilde\varphi, \varphi g\rangle}_\infty < \varepsilon.
\label{unifconv}
\end{equation}
Combining this with~\eqref{step1} yields
\begin{equation}
\lim_{t \to \infty}\left| \mathbb{E}_{\delta_{(r, \upsilon)}}^\uparrow[U_{t + s} | \mathcal{F}_t^\uparrow] - W_\infty^\uparrow \langle \varphi g, \tilde\varphi \rangle\frac{p(r, \upsilon)}{\varphi(r, \upsilon)}\right| = 0.
\end{equation}
The above combined with the conclusion of Lemma \ref{lemma1} gives the conclusion of Theorem \ref{SLLN} along lattice sequences.
\end{proof}

We now make the transition from lattice times to continuous times. 

\begin{proof}[Proof of Theorem~\ref{SLLNbbone} (full sequence)]
For $\varepsilon > 0$ and $(r, \upsilon) \in D \times V$, define
\[
\Omega_\varepsilon(r, \upsilon) \coloneqq \left\{(r', \upsilon') \in D\times V \,: \, g(r', \upsilon')\frac{\varphi(r', \upsilon')}{p(r', \upsilon')} \geq (1+\varepsilon)^{-1}g(r, \upsilon)\frac{\varphi(r, \upsilon)}{p(r, \upsilon)}\right\}.
\]
If we consider the equation \eqref{nonlinear} for the special setting of the NBP, we can decompose it over the first scatter event, rather than the first fission event, from which we will obtain 
{\color{black}\[
w(r, \upsilon) = \hat{\U}_t[w](r, \upsilon) + \int_0^t\U_s\left[ \bS w +G[w]\right](r, \upsilon) \D s, \qquad t\geq 0, x\in E,
\]}
where the semigroup $(\U_t, t\geq 0)$ was defined in \eqref{adv}, $(\hat{\U}_t, t\geq 0)$ was defined in \eqref{advhat}, and the scattering operator $\bS$ was defined in \eqref{S}.
This implies that, for a given $r\in D$ and $\upsilon\in V$, $w(r+\upsilon t, \upsilon)$, and hence $p(r+\upsilon t, \upsilon)$, are continuous for all $t$ sufficiently small. Similarly noting that $\psi_t[\varphi] = {\rm e}^{\lambda_*t}\varphi$, from \eqref{mild}, we can also deduce a similar continuity property of $\varphi$. Hence, together with the assumed directional continuity of $g$, for each $r\in D$, $\upsilon\in V$ and $\varepsilon\ll 1$, there exists a $\delta_\varepsilon$ such that $(r+\upsilon t,\upsilon)\in \Omega_\varepsilon(r, \upsilon) $ for all $t\leq \delta_\varepsilon$.
\smallskip

Next, for each $\delta > 0$ define
\[
\Xi^{\delta, \varepsilon}(r, \upsilon) \coloneqq \mathbf{1}_{\{{\rm supp}(X_t^\uparrow) \, \subset \, \Omega_\varepsilon(r, \upsilon) \text{ for all } t \in [0, \delta]\}}, \quad (r, \upsilon) \in D \times V,
\]
and let $\eta^{\delta, \varepsilon}(r, \upsilon) = \mathbb{E}_{\delta_{(r, \upsilon)}}^\uparrow[\Xi^{\delta, \varepsilon}(r, \upsilon)]\leq 1$.
Appealing to Fatou's Lemma and the continuity properties discussed above, we have, for  $\varepsilon\ll 1$, 
\begin{align*}
\liminf_{\delta\downarrow0}\eta^{\delta, \varepsilon}(r, \upsilon)&\geq \mathbb{E}^\uparrow_{\delta_{(r,\upsilon)}}[\liminf_{\delta\downarrow0}\mathbf{1}_{\{{\rm supp}(X_t^\uparrow) \, \subset \, \Omega_\varepsilon(r, \upsilon) \text{ for all } t \in [0, \delta]\}}]\\
& = \mathbb{E}^\uparrow_{\delta_{(r,\upsilon)}}[\lim_{\delta\downarrow0}\mathbf{1}_{\{(r + \upsilon t, \upsilon)\in\Omega_\varepsilon(r, \upsilon) \text{ for all } t \in [0, \delta]\}}]\\
&=1.
\end{align*}

Since we can effectively see the skeleton as producing at least one\footnote{Although a subtle point in the argument, this is fundamentally the reason why the skeletal decomposition is needed and makes the proof much easier than otherwise.} offspring at every fission event (see also the discussion in Remark \ref{1or2}), it follows that if $t \in [n\delta, (n+1)\delta)$ then,
\begin{align}
&{\rm e}^{-\lambda_* t}\langle g{\varphi}/{p}, X^\uparrow_t \rangle \notag\\
&\ge \frac{{\rm e}^{-\delta}}{(1+\varepsilon)}\sum_{i = 1}^{N_{n\delta}}{\rm e}^{-\lambda_* n\delta}g(R_i(n\delta), \Upsilon_i(n\delta))\frac{\varphi(R_i(n\delta), \Upsilon_i(n\delta))}{p(R_i(n\delta), \Upsilon_i(n\delta))}\Xi^{\delta, \varepsilon}(R_i(n\delta), \Upsilon_i(n\delta)).
\label{estimate}
\end{align}
If we denote the summation on the right-hand side of the above equation by $\tilde{U}_{n\delta}(r, \upsilon)$, and assume that ${\rm supp} (g)$ is compactly embedded in $D$,   then we can apply similar arguments to those given in the proof of Lemma~\ref{lemma1} together with \eqref{identity1} to show that
\begin{align}
\sum_{n = 1}^\infty &\mathbb{E}_{\delta_{(r, \upsilon)}}^\uparrow\left[|\tilde{U}_{n\delta} - \mathbb{E}^\uparrow[\tilde{U}_{n\delta} | \mathcal{F}_{n\delta}^\uparrow] |^2\right] \notag \\
& \hspace{0.5cm}\le C\sum_{n = 1}^\infty {\rm e}^{-\lambda_* n\delta q}\mathbb{E}_{\delta_{(r, \upsilon)}}^\uparrow\left[\sum_{i = 1}^{N_{n\delta}} g(R_i(n\delta), \Upsilon_i(n\delta))^2\frac{\varphi(R_i(n\delta), \Upsilon_i(n\delta))^2}{p(R_i(n\delta), \Upsilon_i(n\delta))^2} \eta^{\delta, \varepsilon}(R_i(n\delta), \Upsilon_i(n\delta))\right] \notag \\
&\hspace{2cm}\le C\sum_{n = 1}^\infty {\rm e}^{-2\lambda_* n\delta }\mathbb{E}_{\delta_{(r, \upsilon)}}^\uparrow \left[\langle (g{\varphi}/{p})^2, X_{n\delta}^\uparrow\rangle\right]\notag\\
&\hspace{2cm}\le \frac{ C}{p(r, \upsilon)}\sum_{n = 1}^\infty {\rm e}^{-2\lambda_* n\delta }\mathbb{E}_{\delta_{(r, \upsilon)}} \left[\langle (g\varphi)^2 p^{-1}, X_{n\delta}\rangle\right]  \notag\\
&\hspace{2cm}= \frac{C}{p(r, \upsilon)}\sum_{n = 1}^\infty{\rm e}^{-2\lambda_*n\delta }\psi_{n\delta}[(g\varphi)^2 p^{-1}](r,\upsilon).
\label{UB1}
\end{align}
Note in particular that the compact embedding of the support of $g$ in $D\times V$ together with %\eqref{<1}
Lemma \ref{Simonlemma}, the fact that $p\leq 1$, $\varphi$ belongs to $L^+_\infty(D\times V)$ and is bounded away from $0$ on compactly embedded subsets of $D\times V$ ensures that $(g\varphi)^2 p^{-1}$ is uniformly bounded away from $0$ and $\infty$ and hence, taking account of the conclusion of Theorem \ref{PF}, the expectation on the right-hand side of \eqref{UB1} is finite.
\smallskip

Noting that 
\[
\mathbb{E}^\uparrow[\tilde{U}_{n\delta} | \mathcal{F}_{n\delta}^\uparrow] = {\rm e}^{-\lambda_*n\delta}\langle g \varphi\eta^{\delta, \varepsilon}/p, X_{n\delta}^\uparrow\rangle,
\]
the consequence of \eqref{UB1}, when taken in the light of the  Borel-Cantelli Lemma and the already proved limit \eqref{onlaticefirst} on lattice times, means that, $\P_{\delta_{(r,\upsilon)}}$-almost surely,
\[
\liminf_{t \to \infty}{\rm e}^{-\lambda_* t}\langle g{\varphi}/{p}, X_t^\uparrow \rangle \ge \frac{{\rm e}^{-\delta}}{1+\varepsilon}\langle g {\varphi}\eta^{\delta, \varepsilon}/{p}, \tilde\varphi p\rangle W_\infty^\uparrow \frac{\varphi(r, \upsilon)}{p(r, \upsilon)}.
\]
Letting $\delta \downarrow 0$ with the help of Fatou's Lemma and then $\varepsilon \downarrow 0$ in the above inequality yields
\begin{equation}
\liminf_{t \to \infty}{\rm e}^{-\lambda_* t}\langle g{\varphi}/{p}, X_t^\uparrow \rangle \ge \langle g\varphi, \tilde\varphi \rangle W_\infty^\uparrow \frac{\varphi(r, \upsilon)}{p(r, \upsilon)},
\label{lower1}
\end{equation}
$\P_{\delta_{(r,\upsilon)}}$-almost surely.  Now replacing  $g$ by   $hp/\varphi$, ensuring still that the support of $h$ is compactly embedded in $D \times V$, so that $hp/\varphi$  is uniformly bounded away from 0 and $\infty$, the lower bound \eqref{lower1} yields
\begin{equation}
\liminf_{t \to \infty}{\rm e}^{-\lambda_* t}\langle h, X_t^\uparrow \rangle \ge \langle h, \tilde\varphi  p\rangle W_\infty^\uparrow \frac{\varphi(r, \upsilon)}{p(r, \upsilon)}.
\label{lower2}
\end{equation}
We can push \eqref{lower2} a little bit further by removing the requirement that the support of $h$  is compactly embedded in $D \times V$. Indeed, suppose that, for $n\geq 1$, $h_n = h\mathbf{1}_{B_n}$, where $h\leq c \varphi/p$ for some constant $c>0$ and  $B_n$ is an increasing sequence of compactly embedded  domains in $D\times V$, such that $\cup_{n\geq 1}B_n= D\times V$.
Then \eqref{lower2} and  together with monotonicity gives us
\begin{align}
\liminf_{t \to \infty}{\rm e}^{-\lambda_* t}\langle h, X_t^\uparrow \rangle &\ge 
\lim_{n\to\infty}
\liminf_{t \to \infty}{\rm e}^{-\lambda_* t}\langle h_n, X_t^\uparrow \rangle \notag\\
&\ge
\lim_{n\to\infty}
\langle h_n, \tilde\varphi p\rangle W_\infty^\uparrow \frac{\varphi(r, \upsilon)}{p(r, \upsilon)} \notag\\
&=\langle h, \tilde\varphi p\rangle W_\infty^\uparrow \frac{\varphi(r, \upsilon)}{p(r, \upsilon)}
\label{liminf2}
\end{align} 
$\P_{\delta_{(r,\upsilon)}}$-almost surely.
\smallskip

To complete the proof of Theorem \ref{SLLNbbone} it now suffices to show that, $\P_{\delta_{(r,\upsilon)}}$-almost surely,  $\limsup_{t \to \infty}{\rm e}^{-\lambda_* t}\langle g, X_t^\uparrow \rangle \leq      \langle g, \tilde\varphi p\rangle W_\infty^\uparrow {\varphi(r, \upsilon)}/{p(r, \upsilon)}.$
To this end note that, for  $0\leq g \leq c \varphi/p$, for some constant $c>0$ (which, without loss of generality, we may take equal to 1), 
\begin{align*}
\limsup_{t \to \infty}{\rm e}^{-\lambda_* t}\langle g, X_t^\uparrow \rangle  &= \limsup_{t \to \infty}\left(\frac{\varphi(r, \upsilon)}{p(r, \upsilon)}W_t^\uparrow -  {\rm e}^{-\lambda_* t}\langle \varphi/p - g , X_t^\uparrow\rangle\right) \\
&= \frac{\varphi(r, \upsilon)}{p(r, \upsilon)}W_\infty^\uparrow - \liminf_{t \to\infty}{\rm e}^{-\lambda_* t}\langle \varphi/p - g , X_t^\uparrow\rangle\\
& \le \frac{\varphi(r, \upsilon)}{p(r, \upsilon)}W_\infty^\uparrow - \langle \varphi/p - g, \tilde\varphi p\rangle\frac{\varphi(r, \upsilon)}{p(r, \upsilon)}W_\infty^\uparrow \\
&= \langle g, \tilde\varphi p\rangle W_\infty^\uparrow \frac{\varphi(r, \upsilon)}{p(r, \upsilon)},
\end{align*}
as required, where we have used the normalisation $\langle\varphi,\tilde\varphi\rangle = 1$.
\end{proof}

\section{Proof of Theorem \ref{SLLN}}
The proof we will give relies on the stochastic embedding of the skeleton 
process $X^\uparrow$ in $X$ together with a measure theoretic trick. It is worth stating the latter in the format of a proposition which is essentially taken from~\cite{HR}. (The reader may note that there is a slight variation in the statement as the original version %in~\cite{HR} 
was missing an additional condition.) 
\begin{prop}\label{MattSimon}
Let $(\Omega, \mathcal{F}, (\mathcal{F}_t,t\geq 0), \mathbb{P})$ be a filtered probability space and define $\mathcal{F}_\infty \coloneqq \sigma(\cup_{i = 1}^\infty\mathcal{F}_t)$. Suppose $(U_t, t\geq 0)$ is an $\mathcal{F}$-measurable non-negative process such that $\textstyle{\sup_{t\geq 0}U_t}$ has finite expectation and  $(\mathbb{E}(U_t | \mathcal{F}_t), t\geq0)$ is c\`adl\`ag. If 
\[
\lim_{t\to\infty}\mathbb{E}(U_t | \mathcal{F}_\infty) = Y, \text{ a.s,}
\]
then
\[
\lim_{t\to\infty}\mathbb{E}(U_t | \mathcal{F}_t) =Y, \text{ a.s.}.
\]
\end{prop}
In fact, this result can be readily obtained by considering $Y_t:=\E(U_t|\mathcal{F}_\infty)$ then 
using right continuity and \emph{Hunt's Lemma}: If $Y_n\rightarrow Y$ a.s., 
$(Y_n, n\in\N)$ is dominated by $\textstyle{\sup_{n\in \N}| Y_n|}$ with $\textstyle{\E\sup_{n\in \N} |Y_n|<\infty}$, 
then $\E(Y_n|\mathcal F_n) \rightarrow \E(Y|\mathcal F_\infty)$ a.s..
\smallskip

We will take the quantities in the above proposition from their definition in the context of the physical process of the neutron transport equation. 
In a similar fashion to the proof of Theorem~\ref{SLLNbbone}, set $U_t = {\rm e}^{-\lambda_*t}\langle g, X_t^\uparrow\rangle$, for $g\in L^+_\infty(D\times V)$, and recall that  $(\mathcal{F}_t, t \ge 0)$ is the filtration generated by the neutron branching process $(X_t, t \ge 0)$. Note that 
we can easily bound $(U_t,t\geq 0)$ by  a multiple of $(W^\uparrow_t,t\geq 0)$ and hence we automatically get that $\textstyle{\sup_{t\geq 0}U_t}$ has a second, and hence first, moments thanks to Corollary \eqref{freebie}.
Due to Theorem~\ref{SLLNbbone} and the fact that $X_t^\uparrow$ is $\mathcal{F}_\infty$-measurable, $U_t = \mathbb{E}(U_t | \mathcal{F}_\infty) $ and hence
\[
\lim_{t\to\infty}\mathbb{E}(U_t | \mathcal{F}_\infty) = \langle gp, \tilde\varphi\rangle W_\infty^\uparrow \frac{\varphi(r, \upsilon)}{p(r, \upsilon)}
\]
$\P_{\delta_{(r,\upsilon)}}$-almost surely, for $r\in D$, $\upsilon\in V$.
\smallskip

Using \eqref{identity1}  (which comes from the skeleton embedding  Theorem \ref{skeleton}, cf. Remark \ref{iii}) as we have in the proof of Theorem \ref{SLLNbbone},  we get
\begin{align*}
\mathbb{E}(U_t | \mathcal{F}_t) &= \mathbb{E}({\rm e}^{-\lambda_* t}\langle g, X_t^\uparrow\rangle| \mathcal{F}_t)
= {\rm e}^{-\lambda_* t}\langle g, pX_t\rangle.
\end{align*}

Combining this with Proposition~\ref{MattSimon} yields
\begin{equation}
\lim_{t\to\infty}{\rm e}^{-\lambda_* t}\langle g, pX_t\rangle = \langle gp, \tilde\varphi\rangle W_\infty^\uparrow \frac{\varphi(r, \upsilon)}{p(r, \upsilon)},
\label{gp}
\end{equation}
$\P_{\delta_{(r,\upsilon)}}$-almost surely. If the support of $g$ is compactly embedded in $D \times V$, then we can replace $g$ by $g/p$, with the assurance that the latter is uniformly bounded away from $0$ and $\infty$ (cf. Lemma \ref{Simonlemma}), and \eqref{gp} gives us
\begin{equation}
\lim_{t\to\infty}{\rm e}^{-\lambda_* t}\langle g, X_t\rangle = \langle g, \tilde\varphi\rangle W_\infty^\uparrow \frac{\varphi(r, \upsilon)}{p(r, \upsilon)},
\label{gp2}
\end{equation}
$\P_{\delta_{(r,\upsilon)}}$-almost surely.  We can remove the assumption that the support of $g$ is compactly embedded in $D \times V$ by appealing to similar reasoning as that of the computation in \eqref{liminf2}.

\smallskip

To complete the proof of almost sure convergence, we need to show that $W_\infty^\uparrow/p = W_\infty$, almost surely. To do so, note that if we take $g= \varphi$ in \eqref{gp2},  noting that the left-hand side is equal to $\lim_{t\to\infty}W_t\varphi(r,\upsilon)$ and $\langle\varphi, \tilde\varphi\rangle=1$, we get the desired result.

\smallskip

Finally, for the convergence in $L_2(\P)$, first recall that we already know that $\textstyle{\E(\sup_{t\geq0} W_t^2)<\infty}$ by Doob's $L_p$-inequality and $L_2(\P)$-boundedness of $W$ (see discussion within proof of Lemma \ref{lemma1}). 
Then, by assumption $g\leq \phi$, we similarly have $\textstyle{\sup_{t\geq0} \langle g, X_t \rangle}$ in $L_2(\P)$, hence we can use the dominated convergence theorem to conclude that we have convergence in $L_2(\P)$, as well as almost surely.
\hfill$\square$
\section{Concluding remarks}

The proof of Theorem \ref{SLLN} above gives a generic approach for branching particle systems which have an identified skeletal decomposition. Indeed, the reasoning is robust and will show in any such situation that the existence of a strong law of large numbers for the skeleton implies almost immediately a strong law of large numbers for the original process into which the skeleton is embedded. 
As an exercise, the reader is encouraged to consider the setting of a branching Brownian motion in a strip (cf. \cite{HHK}). Supposing a strong law of large numbers exists on the skeleton there (in that setting it is called the `blue tree'), then we claim that the  the above reasoning applied verbatim will deliver the strong law of large numbers for the branching Brownian motion in a strip.

\smallskip

More generally, we claim that, modulo some minor technical modifications (e.g. taking account of the fact that $E$ may be unbounded),  in the general MBP setting of Theorem \ref{skeleton}, an analogue of Theorem \ref{SLLN} may be reconstructed  once the following three important components are in hand:  (i) An analogue of Theorem \ref{PF}; (ii) A degree of knowledge concerning the continuity properties of $\varphi$ and $p$; (iii) the martingale $W$ has the property $\textstyle{\mathbb{E}_{\delta_x}[\sup_{t\geq 0}W_t^2]<\infty}$, for all $x\in E$. Indeed, last of these three may be weakened to $\gamma$-integrability of the martingale $W$, for $\gamma\in(1,2)$, in which case one may replace many of the estimates in the Borel-Cantelli arguments by $\gamma$ moment estimates instead of second moment estimates (see e.g. \cite{EHK} for comparison).
\smallskip

It is also worth pointing out however that the reasoning in the proof of Theorem \ref{SLLN} does not so obviously work in the setting of superprocesses with a skeletal decomposition. Indeed a crucial step, which is automatic for branching particle systems, but less obvious for superprocesses, is the point in the argument at which we claim that $U_t = \mathbb{E}(U_t | \mathcal{F}_\infty) $. In the particle system, this statement follows immediately from the fact that $\mathcal{F}_\infty$ carries enough information to construct the marks $\uparrow$ and $\downarrow$ on particles because individual genealogical lines of descent are identifiable. For superprocesses, it is less clear how to choose the filtration $(\mathcal{F}_t, t\geq 0)$ so that the notion of genealogy or otherwise can be used to claim that $X^\uparrow_t$, and hence  $U_t$, is $\mathcal{F}_\infty$-measurable.

\section*{Acknowledgements} The body of work in this article as well as \cite{SNTE, MultiNTE} was born out of a surprising connection that was made at the problem formulation ``Integrative Think Tank'' as part of the EPSRC Centre for Doctoral Training SAMBa in the summer of 2015. We are indebted to Professor Paul Smith and Dr. Geoff Dobson from the ANSWERS modelling group at Wood for the extensive discussions as well as hosting at their offices in Dorchester.  We are also grateful to Denis Villemonais for discussions on general convergence theorems for semigroups.
We are also grateful to a referee and the AE who made a number of very helpful suggestions.

\bibliography{references}{}

\begin{thebibliography}{10}

\bibitem{AH}
S.~Asmussen and H.~Hering.
\newblock Strong limit theorems for general supercritical branching processes
  with applications to branching diffusions.
\newblock {\em Z. Wahrscheinlichkeitstheorie und Verw. Gebiete},
  36(3):195--212, 1976.

\bibitem{BKM}
J.~Berestycki, A.~E. Kyprianou, and A.~Murillo-Salas.
\newblock The prolific backbone for supercritical superprocesses.
\newblock {\em Stochastic Process. Appl.}, 121(6):1315--1331, 2011.

\bibitem{CRSZ}
Z-Q. Chen, Y-X. Ren, R.~Song, and R.~Zhang.
\newblock Strong law of large numbers for supercritical superprocesses under
  second moment condition.
\newblock {\em Front. Math. China}, 10(4):807--838, 2015.

\bibitem{CRY}
Z-Q. Chen, Y-X. Ren, and T.~Yang.
\newblock Law of large numbers for branching symmetric {H}unt processes with
  measure-valued branching rates.
\newblock {\em J. Theoret. Probab.}, 30(3):898--931, 2017.

\bibitem{MultiNTE}
A.~M.~G. Cox, S.~C. Harris, E.~L. Horton, and Andreas~E. Kyprianou.
\newblock Multi-species neutron transport equation.
\newblock {\em J. Stat. Phys.}, 176(2):425--455, 2019.

\bibitem{MCNTE}
A.~M.~G. Cox, S.C. Harris, E.~Horton, A.~E. Kyprianou, and M.~Wang.
\newblock Monte carlo methods for the neutron transport equation.
\newblock {\em Working document}.

\bibitem{SNTE-III}
A.~M.~G. Cox, E.~Horton, A.~E. Kyprianou, and D.~Villemonais.
\newblock Stochastic methods for neutron transport equation {III}: Generational
  many-to-one and $ k_{\texttt{eff} } $.
\newblock {\em Preprint}, 2019.

\bibitem{D}
R.~Dautray, M.~Cessenat, G.~Ledanois, P.-L. Lions, E.~Pardoux, and R.~Sentis.
\newblock {\em M\'ethodes probabilistes pour les \'equations de la physique}.
\newblock Collection du Commissariat a l'\'energie atomique. Eyrolles, Paris,
  1989.

\bibitem{DL6}
R~Dautray and J.-L. Lions.
\newblock {\em Mathematical analysis and numerical methods for science and
  technology. {V}ol. 6}.
\newblock Springer-Verlag, Berlin, 1993.
\newblock Evolution problems. II, With the collaboration of Claude Bardos,
  Michel Cessenat, Alain Kavenoky, Patrick Lascaux, Bertrand Mercier, Olivier
  Pironneau, Bruno Scheurer and R\'emi Sentis, Translated from the French by
  Alan Craig.

\bibitem{DW07}
T.~Duquesne and M.~Winkel.
\newblock Growth of {L}\'evy trees.
\newblock {\em Probab. Theory Rel. Fields}, 139(3-4):313--371, 2007.

\bibitem{Dynkin2}
E.~B. Dynkin.
\newblock {\em Diffusions, superdiffusions and partial differential equations},
  volume~50 of {\em American Mathematical Society Colloquium Publications}.
\newblock American Mathematical Society, Providence, RI, 2002.

\bibitem{Dynkin1}
E.~B. Dynkin.
\newblock {\em Superdiffusions and positive solutions of nonlinear partial
  differential equations}, volume~34 of {\em University Lecture Series}.
\newblock American Mathematical Society, Providence, RI, 2004.
\newblock Appendix A by J.-F. Le Gall and Appendix B by I. E. Verbitsky.

\bibitem{EKW}
M.~Eckhoff, A.~E. Kyprianou, and M.~Winkel.
\newblock Spines, skeletons and the strong law of large numbers for
  superdiffusions.
\newblock {\em Ann. Probab.}, 43(5):2545--2610, 2015.

\bibitem{EHK}
J.~Engl{\"a}nder, S.~C. Harris, and A.~E. Kyprianou.
\newblock Strong law of large numbers for branching diffusions.
\newblock {\em Ann. Inst. Henri Poincar\'e Probab. Stat.}, 46(1):279--298,
  2010.

\bibitem{HHK}
S.~C. Harris, M.~Hesse, and A.~E. Kyprianou.
\newblock Branching {B}rownian motion in a strip: survival near criticality.
\newblock {\em Ann. Probab.}, 44(1):235--275, 2016.

\bibitem{HR}
S.~C. Harris and M.~I. Roberts.
\newblock A strong law of large numbers for branching processes: almost sure
  spine events.
\newblock {\em Electron. Commun. Probab.}, 19:no. 28, 6, 2014.

\bibitem{H}
T.~E. Harris.
\newblock {\em The theory of branching processes}.
\newblock Dover Phoenix Editions. Dover Publications, Inc., Mineola, NY, 2002.
\newblock Corrected reprint of the 1963 original [Springer, Berlin; MR0163361
  (29 \#664)].

\bibitem{SNTE}
E.~Horton, A.~E. Kyprianou, and D.~Villemonais.
\newblock Stochastic methods for the neutron transport equation {I}: Linear
  semigroup asymptotics.
\newblock {\em To appear in Annals of Applied Probability}.

\bibitem{INW1}
N.~Ikeda, M.~Nagasawa, and S.~Watanabe.
\newblock Branching {M}arkov processes. {I}.
\newblock {\em J. Math. Kyoto Univ.}, 8:233--278, 1968.

\bibitem{INW2}
N.~Ikeda, M.~Nagasawa, and S.~Watanabe.
\newblock Branching {M}arkov processes. {II}.
\newblock {\em J. Math. Kyoto Univ.}, 8:365--410, 1968.

\bibitem{INW3}
N.~Ikeda, M.~Nagasawa, and S.~Watanabe.
\newblock Branching {M}arkov processes. {III}.
\newblock {\em J. Math. Kyoto Univ.}, 9:95--160, 1969.

\bibitem{LPS}
B.~Lapeyre, \'E. Pardoux, and R.~Sentis.
\newblock {\em Introduction to {M}onte-{C}arlo methods for transport and
  diffusion equations}, volume~6 of {\em Oxford Texts in Applied and
  Engineering Mathematics}.
\newblock Oxford University Press, Oxford, 2003.
\newblock Translated from the 1998 French original by Alan Craig and Fionn
  Craig.

\bibitem{MT}
S.~Maire and D.~Talay.
\newblock On a {M}onte {C}arlo method for neutron transport criticality
  computations.
\newblock {\em IMA J. Numer. Anal.}, 26(4):657--685, 2006.

\bibitem{M-K}
M.~Mokhtar-Kharroubi.
\newblock {\em Mathematical topics in neutron transport theory}, volume~46 of
  {\em Series on Advances in Mathematics for Applied Sciences}.
\newblock World Scientific Publishing Co., Inc., River Edge, NJ, 1997.
\newblock New aspects, With a chapter by M. Choulli and P. Stefanov.

\bibitem{MWY}
T.~Mori, S.~Watanabe, and T.~Yamada.
\newblock On neutron branching processes.
\newblock {\em Publ. Res. Inst. Math. Sci.}, 7:153--179, 1971/72.

\bibitem{MSP}
A.~Murillo-Salas and J.~L. P\'{e}rez.
\newblock The backbone decomposition for superprocesses with non-local
  branching.
\newblock In {\em X{I} {S}ymposium on {P}robability and {S}tochastic
  {P}rocesses}, volume~69 of {\em Progr. Probab.}, pages 199--216.
  Birkh\"{a}user/Springer, Cham, 2015.

\end{thebibliography}
\bibliographystyle{plain}

\begin{table}[h]
%\caption{Index of some notation.}
%\begin{tabularx}{\textwidth}{l l l}
\begin{tabular}{l l l}
&\multicolumn{1}{c}{\bf Glossary of some commonly used notation}&\\
&\multicolumn{1}{c}{\bf (Th. = Theorem, a. = above, b. = below)}&\\
&&\\
\hline
Notation & Description & Introduced\\
\hline
&&\\
$(\psi_t, t\geq 0)$ & Solution to mild NTE/NBP expectation 
semigroup& \eqref{semigroup}, \eqref{mild}\\
$D$ and $V$ & Physical and velocity domain & \S \ref{intro}\\
$\sigma_{\texttt{s}}$, $\sigma_{\texttt{f}}$ and $\sigma$ & Scatter, fission and total cross-sections 
&b. \eqref{bNTE}\\
$\pi_{\texttt s}$ and $\pi_{\texttt f}$ & Scatter and fission kernels & b. \eqref{bNTE}\\
$\bS$ and $\bF$ & Scatter and fission operators & \eqref{S}, \eqref{F}\\

$n_\texttt{max}$ & Maximum number of neutrons in a fission event& b. \eqref{Erv}\\
%$\underline{\beta}$ &Minimal growth rate & \eqref{underbeta}\\
$\lambda_*$, $\varphi$ and $\tilde\varphi$& Leading eigenvalue, right- and left-eigenfunctions &Th. \ref{PF}\\
$(W_t, t\geq 0)$& Additive martingale &\eqref{addmg}\\

&&\\
\hline
&&\\
$E$ and $\dagger$ & Domain and cemetery state on which $\bP$ and $\xi$ is defined &\S \ref{MBP}\\
$\bP$ and $\hat\bP$ &  Particle motion semigroup on $E$  and $E\cup\{\dagger\}$ resp.& \S \ref{MBP}\\
$\bL$ & Generator associated to $\bP$ in the setting of NBP &\eqref{L}\\
$(\xi, \mathbf{P}_x)$ & Markov process issued from $x\in E$ whose semigroup is $\bP$ & \eqref{PtoP}\\
$(X, \mathbb{P}_\mu)$ & General $(\bP, G)$-MBP (and NBP) when issued from $\mu$ & \S \ref{MBP} (and \eqref{atomicvalued})\\
$(u_t, t\geq 0)$ & Non-linear semigroup of $X$ (and NBP) & \eqref{MBPnonlinear} (and \eqref{ut})\\
$\zeta$ &Lifetime of $X$ & \eqref{zeta}\\

$\varsigma(x)$ & Instantaneous branching rate of $X$ at $x\in E$& \S \ref{MBP}\\
$\mathcal{P}_x$  & Offspring law  of $X$ when parent at  $x\in E$ (and for NBP)& a. \eqref{PtoP} (and \eqref{Erv})\\
$G$ & Branching generator (and for NBP) & \eqref{similar111} (and \eqref{G})\\
$(x_i, i = 1,\cdots, N)$& Position and number of offspring positions  of a family in $X$ & \S \ref{MBP}\\
$w(x)$ (resp.  $p(x)$)& Prob. extinction (resp. surivival) when issued from $x\in E$&  \eqref{wdef} (resp.  \eqref{pdef})\\
&&\\
\hline
&&\\
$(X^\downarrow, \mathbb{P}^\downarrow_\mu)$ &MBP conditioned to die out and law when issued from $\mu$ & Th. \ref{skeleton} (i)\\
$(u^\downarrow_t, t\geq 0)$ &Non-linear semigroup of $X^\downarrow$ & \eqref{rednonlinearNTE}, \eqref{udown}\\

$\bP^\downarrow$ and $\hat\bP^\downarrow$ & Markov semigroup  associated to $X^\downarrow$ on $E$ and $E\cup\{\dagger\}$ resp. &Th. \ref{skeleton} (i)\\
$(\xi, \mathbf{P}_x^\downarrow)$ & Markov process  associated to $\bP^\downarrow$ issued from $x\in E$&\eqref{COMdown}\\
$\bL^\downarrow$ & Generator associated to $\bP^\downarrow$ in the setting of NBP &\eqref{Ldown}\\
$\varsigma^\downarrow(x)$ & Instantaneous branching rate of $X^\downarrow$ at $x\in E$&  \eqref{redbrate}  \\
$\mathcal{P}^\downarrow_x$ &Offspring law  of $X^\downarrow$ when parent at  $x\in E$ & \eqref{COMcalP}\\
$G^\downarrow$  & Branching generator of $X^\downarrow$  
&\eqref{redbmech} \\
$(x^\downarrow_i, i = 1,\cdots, N^\downarrow)$& Position and number of offspring positions  of a family in  $X^\downarrow$& Th. \ref{skeleton} (ii)\\

&&\\
\hline
&&\\
$(X^\uparrow, \mathbb{P}^\uparrow_\mu)$ &Skeleton MBP ($X$ conditioned to survive) when issued from $\mu$ & Th. \ref{skeleton} (ii)\\
$(X^\updownarrow, \mathbb{P}^\updownarrow_\mu)$ &Skeleton $X^\uparrow$ dressed with $X^\downarrow$ trees when issued from $\mu$ & Th. \ref{skeleton} (ii), \eqref{updownCOM}\\
$(u^\updownarrow_t, t\geq 0)$ &Non-linear semigroup of $X^\updownarrow$ & \eqref{uupdown},  \eqref{updownevolve}\\
$\bP^\uparrow$ & Markov semigroup  associated to $X^\uparrow$  &Th. \ref{skeleton} (ii)\\
$(\xi, \mathbf{P}_x^\uparrow)$ & Markov process  associated to $\bP^\uparrow$ issued from $x\in E$&\eqref{pmg}\\
$\bL^\uparrow$ & Generator associated to $\bP^\uparrow$ in the setting of NBP &\eqref{Lup}\\
 $\varsigma^\updownarrow(x)$ & Instantaneous branching rate of $X^\uparrow$ and $X^\updownarrow$ at $x\in E$&   \eqref{bluebrate} \\
$\mathcal{P}^\updownarrow_x$ &Joint $\uparrow$ and $\downarrow$ offspring law of $X^\updownarrow$ when parent at  $x\in E$ & \eqref{PcalupdownCOM}\\

 $G^\uparrow$ & Branching generator of $X^\uparrow$
& \eqref{bluebmech} \\
$(x^\uparrow, i = 1,\cdots, N^\uparrow)$& Position and number of offspring positions  of a family in  $X^\uparrow$& Th. \ref{skeleton} (ii)\\
$G^\updownarrow$ & Joint branching generator of $\uparrow$-type and $\downarrow$-type in $X^\updownarrow$  
& \eqref{Gupdown0}, \eqref{Gupdown}\\
&&\\
\hline
\end{tabular}
\label{table-notation}
\end{table}

\end{document}